\documentclass{amsart}
\usepackage{amssymb, latexsym, euscript}
\usepackage[usenames,dvipsnames,svgnames,table]{xcolor}
\def\CR{\color{red} }

\def\sD{{\mathfrak D}}      
   \def\sH{{\mathfrak H}}   
   \def\sK{{\mathfrak K}}

\def\st{{\mathfrak t}}
\def\sh{{\mathfrak h}}
\def\ss{{\mathfrak s}}

\def\sr{{\mathfrak r}}
\def\su{{\mathfrak u}}
\def\sb{{\mathfrak b}}

\def\su{{\mathfrak u}}

      \def\dC{{\mathbb C}}

   \def\dN{{\mathbb N}}   
      \def\dR{{\mathbb R}}

\def\bB{{\mathbf B}}
\def\bL{{\mathbf L}}

\def\bF{{\mathbf F}}

\def\wt#1{{{\widetilde #1} }}

\def\bm\chi{\mbox{\boldmath$\chi$}}
\def\half{{\frac{1}{2}}}

\def\col{{\rm col\,}}

\def\ker{{\rm ker\,}}

\def\ran{{\rm ran\,}}
\def\cran{{\rm \overline{ran}\,}}
\def\dom{{\rm dom\,}}
\def\mul{{\rm mul\,}}

\def\cdom{{\rm \overline{dom}\,}}
\def\clos{{\rm clos\,}}

\let\xker=\ker \def\ker{{\xker\,}}

\def\uphar{{\upharpoonright\,}}

\newtheorem{theorem}{Theorem}[section]
\newtheorem{proposition}[theorem]{Proposition}
\newtheorem{corollary}[theorem]{Corollary}
\newtheorem{lemma}[theorem]{Lemma}
\newtheorem{definition}[theorem]{Definition}

\theoremstyle{definition}

\newtheorem{remark}[theorem]{Remark}

\numberwithin{equation}{section}

\begin{document}
{\CR
\title[Representing maps for semibounded forms]
{Representing maps for semibounded forms \\
and their Lebesgue type decompositions}
}

\author[S.~Hassi]{S.~Hassi}
\author[H.S.V.~de~Snoo]{H.S.V.~de~Snoo}

\address{Department of Mathematics and Statistics \\
University of Vaasa \\
P.O. Box 700, 65101 Vaasa \\
Finland} \email{sha@uwasa.fi}

\address{Bernoulli Institute for Mathematics, Computer Science and Artificial Intelligence \\
University of Groningen \\
P.O. Box 407, 9700 AK Groningen \\
Nederland}
\email{h.s.v.de.snoo@rug.nl}

\date{}
\thanks{}

%\translator{}

%\date{ \today; Filename: \jobname.}


\begin{abstract}
For a semibounded sesquilinear form $\st$ in a Hilbert space $\sH$
there exists a representing map $Q$ from $\sH$ to another Hilbert space $\sK$,
such that $\st[\varphi, \psi]-c(\varphi, \psi)\!=\!(Q\varphi,Q\psi)$,
$\varphi,\psi \in \dom \st$, with $c \in \dR$ a lower bound of $\st$.
Representing maps offer a simplifying tool to study general semibounded forms.
By means of representing maps closedness, closability, and singularity of $\st$
are immediately translated into the corresponding properties of the operator $Q$, and vice versa.
Also properties of sum decompositions $\st=\st_1+\st_2$ of
a nonnegative form $\st$ with two other nonnegative forms $\st_1$ and $\st_2$ in $\sH$
can be analyzed by means of associated nonnegative contractions $K\in\bB(\sK)$.
This helps, for instance, to establish an explicit operator theoretic characterization
for the summands $\st_1$ and $\st_2$ to be, or not to be, mutually singular.
Such sum decompositions are used to study characteristic properties of the so-called Lebesgue type
decompositions of semibounded forms $\st$, where $\st_1$ is closable and $\st_2$ singular;
in particular, this includes the Lebesgue decomposition of a semibounded form due to B. Simon.
Furthermore, for a semibounded form $\st$ with its representing map $Q$
it will be shown that the corresponding semibounded selfadjoint relation $Q^*Q^{**} +c$
is uniquely determined by a limit version of the classical representation theorem for the form $\st$,
being studied by W. Arendt and T. ter Elst in a sectorial context.
Via representing maps a full treatment is given of the convergence of monotone
sequences of semibounded forms.
\end{abstract}

\maketitle

\section{Introduction}%\label{sect1}

Let $\st \in \bF(\sH)$, the sesquilinear forms in the Hilbert space $\sH$,
and assume that $\st$ is semibounded.
Then with $c \in \dR$ a lower bound of $\st$ there exists a representing map for $\st-c$,
i.e., a linear operator $Q \in \bL(\sH, \sK)$,
such that
\begin{equation}\label{uksi}
 \st[\varphi, \psi]-c(\varphi, \psi)
 =(Q\varphi, Q \psi), \quad \varphi, \psi \in \dom \st=\dom Q.
\end{equation}
Here $Q \in \bL(\sH, \sK)$ stands for the space of all linear relations from $\sH$ to a Hilbert space $\sK$.
Without loss of generality it may be assumed that
the representation in \eqref{uksi} is minimal in the sense that $\ran Q$ is dense in $\sK$.
With such a representing map several notions for linear operators
are connected to similar notions for semibounded forms:
think of closability and singularity.
For instance, a semibounded form $\st$ is closable
if and only if the representing map $Q$ is closable and, in this case,
the closure $\bar{\st}$ of the form $\st$  is given by
\[%begin{equation}\label{kaksi}
 \bar{\st}[\varphi, \psi]-c(\varphi, \psi)
 =(Q^{**}\varphi, Q^{**} \psi), \quad \varphi, \psi \in \dom \st=\dom Q,
\]%end{equation}
where the operator $Q^{**}$ is the closure of $Q$.
Likewise, a nonnegative form $\st$ is  singular
if and only if the representing map $Q$ is singular.
A short overview of representing maps is given in Section \ref{sec2}.
Recall that representing maps have
Lebesgue type decompositions based on
a decomposition of Hilbert spaces going back
to de Branges and Rovnyak;
general linear relations were decomposed along
these lines  in \cite{HS2023a}.
In the present paper the representing maps will be used
to decompose semibounded
forms in semibounded closable and  nonnegative singular forms
and to offer a new interpretation of the
classical representation theorem for forms.
Furthermore, this interpretation will play
a role when considering the convergence of monotone sequences of forms.
The contents of the present paper will now be described.

\medskip

The first topic is the Lebesgue decomposition of a semibounded form
$\st \in \bF(\sH)$.  In general a semibounded form $\st \in \bF(\sH)$ need not be closed
or closable, so that the interest is in decomposing $\st$ into a sum where the summands
are closable and singular, respectively. In particular, there exist forms
 $\st_{\rm reg}, \st_{\rm sing} \in \bF(\sH)$ such that
 \begin{equation}\label{leeb}
\st=\st_{\rm reg}+\st_{\rm sing}, \quad  \dom \st=\dom \st_{\rm reg}=\dom \st_{\rm sing},
\end{equation}
where $\st_{\rm reg}$ is semibounded and closable
and $\st_{\rm sing}$ is nonnegative and singular.
Moreover,  this decomposition is uniquely determined since
the component $\st_{\rm reg}$ is the maximal closable part of $\st$
and $\st_{\rm sing}=\st-\st_{\rm reg}$.
This so-called Lebesgue decomposition goes back to Simon \cite{S3}.
Here the decomposition will be shown by means of the Lebesgue decomposition
of the representing map $Q$ of $\st$ in \eqref{uksi}. The regular part $\st_{\rm reg}$
of a form $\st\in \bF(\sH)$ will appear when discussing the representation
of general forms by means of nonnegative selfadjoint relations (see Section \ref{repr})
or describing the limits of monotone sequences of forms (see Section \ref{mono}).

\medskip

The second topic is a description of all Lebesgue type decompositions
of a semibounded form $\st \in \bF(\sH)$ as
\begin{equation}\label{kolmen}
 \st=\st_1+\st_2, \quad \dom \st=\dom \st_1=\dom \st_2,
\end{equation}
where  the form $\st_1 \in \bF(\sH)$ is semibounded and closable
while the form $\st_2 \in \bF(\sH)$ is nonnegative and singular.
In \cite{Kosh} there is a treatment of singular forms with applications.
Note that the Lebesgue decomposition in \eqref{leeb} is an example
of a Lebesgue type decomposition.

The first step to study form sums \eqref{kolmen} is taken in Section \ref{sec3}, where the interest is
in sum decompositions \eqref{kolmen} of a nonnegative form $\st$
with nonnegative forms $\st_1$ and $\st_2$.
The sum decompositions are characterized by means of nonnegative contractions $K \in \bB(\sK)$,
where $\sK$ is the Hilbert space associated with the representing map $Q$ of $\st$.
The main results for the parametrization can be found in Theorem \ref{Lebtype},
while the question of minimality of representation for the sum $\st_1+\st_2$,
assuming minimality of representations of the summands $\st_1$ and $\st_2$,
is dealt with in Theorem \ref{minim}.
It appears that the minimality of the parametric representation for the sum
$\st_1+\st_2$ is intimately connected with the interaction between the components $\st_1$ and $\st_2$.
Following the measure theoretic analog for the mutual singularity of pairs of nonnegative measures,
the parallel sum $\st_1:\st_2$ of nonnegative forms $\st_1,\st_2$ is treated
and described by means of the contraction $K$ in Proposition \ref{repparsum}.
This leads to a classification of form sums $t_1+t_2$ in two main categories:
those for which $\st_1:\st_2=0$, the case of mutual singularity,
and the opposite case $\st_1:\st_2\neq 0$, the case where the components
$\st_1$ and $\st_2$ interact with each other.
The Lebesgue decomposition of $\st$ is of the first category,
cf. \cite[Proposition 2.10]{HSS2009}, and correspond to the case where $K$ is an orthogonal projection.
The other case, where $K$ is not an orthogonal projection means non-minimal
representation for the sum $\st_1+\st_2$.
A deeper operator theoretic treatment to study these phenomena is presented in \cite{HS2023a}.

The second step to study form sums \eqref{kolmen} can be found in Section \ref{sec32}, where
the nonnegative contractions $K \in \bB(\sK)$ are characterized
for which in \eqref{kolmen} the form  $\st_1$ is closable and the form $\st_2$ is singular.
In fact, the general case where $\st$ is semibounded will be presented,
using the translation invariance of closable forms, see Theorem \ref{Lebtypep}.
Moreover, in Theorem \ref{uniq}  the uniqueness of the Lebesgue type
decomposition is characterized.
Lebesgue type decompositions for a pair of nonnegative forms were introduced
in \cite{HSS2009}; in the present paper the emphasis is on single forms.

\medskip

The third topic is the representation of semibounded forms
via semibounded selfadjoint relations.
This is treated in Section \ref{repr}.
Let $\st \in \bF(\sH)$ be a semibounded form
with the representation \eqref{uksi}.
If $\st$ is closed then the well-known representation theorem
asserts that there is a semibounded selfadjoint relation $H$ in $\sH$
with $\dom H \subset \dom \st$ and having the same lower bound as $\st$,
such that for all elements $\{\varphi, \varphi'\} \in H$,
\[
\st[\varphi, \psi]=(  \varphi',  \psi)
\quad \mbox{for all} \quad  \psi \in \dom \st.
\]
In Theorem \ref{thm1} it is shown in the general case, when $\st$
is not necessarily closed, that there exists a unique semibounded selfadjoint relation $A$,
namely $A=Q^*Q^{**}+c$, such that $\dom A \subset \dom \st$ and
for each element $\{\varphi, \varphi'\} \in A$
there exists a sequence $\varphi_n \in \dom \st$ for which
\begin{equation}\label{uksi0}
\varphi_n \to_\st \varphi  \quad \mbox{and} \quad
\st[\varphi_n,\psi] \to (\varphi',\psi) \quad \mbox{for all} \quad \psi\in\dom \st,
\end{equation}
where $\varphi_n \to_\st \varphi$ means
$\varphi_n \to \varphi$ in $\sH$ and $\st[\varphi_n-\varphi_m] \to 0$.
The connection with the representation for the form $\st_{\rm reg}$ can be
found in Theorem \ref{terelst}.
The approach via representing maps gives a simple proof of all these facts.
The representation \eqref{uksi0} for a general semibounded form $\st$
coincides with the one of Arendt and ter Elst \cite{AtE1}.

\medskip

The fourth and last topic is the convergence
of monotone sequences of semibounded forms.
The approach via representing maps will be used  in Section \ref{mono}
to treat the convergence of such sequences
along the lines of Section \ref{repr}.
In the background of this section is the treatment
of the convergence of nondecreasing sequences
of linear operators or relations; cf.  \cite{HS2023seq1}.

\medskip

The point of view in the present paper is to obtain results
for a single semibounded form, as described above,
from the general considerations in \cite{HSS2018}
for linear operators and relations, applied via the representing map of the form.
The topic of Lebesgue type decompositions
has a long history; here it is only mentioned that
the first results for Lebesgue decompositions of forms go back to Simon \cite{S3},
while the related case of Lebesgue type decompositions for
a pair of nonnegative bounded operators was treated
a little earlier by Ando \cite{An}; for more references, see \cite{HSS2018}.
For the representation of semibounded forms by means of selfadjoint operators or relations,
see \cite{Kato}, \cite{Kosh},  \cite{BHS},  \cite{AtE1}.
Closely associated results for the convergence of monotone sequences of forms
can be found in  \cite{Kato},  \cite{RS1}, \cite{S3}.

\medskip

The treatment of semibounded forms in terms of representing maps can also be used
for an interpretation of the work of Sebesty\'en and Stochel \cite{SS} concerning Friedrichs
and Kre\u{\i}n type extensions of a semibounded relation.
An extension of the present Lebesgue type decompositions for semibounded forms
can be given for the case of a pair of nonnegative bounded operators or forms;
this includes Lebesgue type decompositions for the context of \cite{An}.
Furthermore, it should be mentioned that many results in the present paper
remain valid in the context of sectorial forms and sectorial operators or relations.

\section{Representing maps for semibounded sesquilinear forms}\label{sec2}

This section is a summary of the facts concerning semibounded sesquilinear forms in   
a Hilbert space. There is a special emphasis on the notion of a representing map for
a semibounded sesquilinear form.

\medskip

 A sesquilinear form $\st=\st[\cdot, \cdot]$ in a Hilbert space $\sH$ over $\dC$
with an inner product $(\cdot, \cdot)$ is a mapping from  $\sD \times \sD$ to $\dC$,
where $\sD=\dom \st$  is a linear subspace of $\sH$,
which is linear in the first entry and antilinear in the second entry.
In what follows, sesquilinear forms over $\dC$ will be often called just forms,
and the class of forms
in a Hilbert space $\sH$ will be denoted  by $\bF(\sH)$.
Also the diagonal $\st[\cdot]$ will be used:
$\st[\varphi]=\st[\varphi, \varphi]$, $\varphi \in \dom \st$.
The form $\st$ is called symmetric
if $\st[\varphi, \psi]=\overline{\st[\psi, \varphi]}$ for all $\varphi, \psi \in \sD$.
Equivalently, $\st$ is symmetric if and only if $\st[\varphi]\in\dR$ for all $\varphi\in \sD$.
Let $\st_1, \st_2 \in \bF(\sH)$ be forms
 with domains $\dom \st_{1}$ and $\dom \st_{2}$.
Then the sum $\st_{1}+\st_{2}$, defined by
\[%begin{equation}\label{fsu}
 (\st_{1}+\st_{2})[\varphi, \psi]=\st_{1}[\varphi, \psi]+\st_{2}[\varphi, \psi],
 \quad \varphi, \psi \in \dom \st_{1} \cap \dom \st_{2},
\]%end{equation}
belongs to $\bF(\sH)$.
In particular, $\st[\varphi, \psi]+c (\varphi, \psi)$,
with $\varphi, \psi \in \dom \st$ and $c \in \dR$, defines a form, denoted by $\st+c$.
The inequality $\st_{1} \leq  \st_{2}$ for forms $\st_{1}, \st_{2} \in \bF(\sH)$ is defined by
\begin{equation}\label{ineq0}
\dom \st_{2} \subset  \dom \st_{1},  \quad
\st_{1}[\varphi] \leq \st_{2}[\varphi], \quad \varphi \in  \dom \st_{2}.
\end{equation}
 The inclusion $\st_2 \subset \st_1$ means that $\dom \st_2 \subset \dom \st_1$
and $\st_1$ agrees with $\st_2$ on $\dom \st_2$.
In this case $\st_2$ is called a restriction of $\st_1$,
while $\st_1$ is called an extension of $\st_2$.
In particular, $\st_{2} \subset \st_{1}$ implies $\st_{1} \leq \st_{2}$.
All these elementary facts can be found in \cite{BHS, Kato, Kosh}.

\medskip

A form $\st \in \bF(\sH)$ with domain $\dom \st$ is called bounded if
\[%begin{equation}\label{lb0}
 |\st[\varphi]|  \leq a \|\varphi\|^{2}, \quad \varphi \in \dom \st,
\]%end{equation}
for some $a \geq 0$. The form $\st \in \bF(\sH)$
is called bounded from below by $c\in \dR$ if
\begin{equation}\label{lb}
 \st[\varphi] \geq c \|\varphi\|^{2}, \quad \varphi \in \dom \st,
\end{equation}
and nonnegative if $\st[\varphi] \geq 0$, $\varphi \in \dom \st$.
A semibounded form is automatically symmetric.
The sum of forms which are semibounded from below
is also semibounded from below.
The lower bound $m(\st) \in \dR$ of $\st$
is the greatest of all numbers $c$ in \eqref{lb}:
\[%begin{equation}\label{thelb}
  m(\st)= \inf \left\{ \,\frac{\st[\varphi, \varphi]}{(\varphi, \varphi)} :
  \, \varphi \in \dom \st \, \right\}.
\]%end{equation}
For $c \leq m(\st)$  define the form $\st-c$ as follows
\[%begin{equation}\label{lb+}
 (\st-c)[\varphi, \psi]=\st[\varphi, \psi]- c(\varphi, \psi)_{\sH},
 \quad \varphi, \psi \in \dom \st.
\]%end{equation}
Then $\st-c \in \bF(\sH)$ is a nonnegative form
with $\dom (\st-c)=\dom \st$.
The kernel $\ker (\st-c)$ of $\st-c$ is defined by
\[
 \ker (\st-c)=\{ \varphi \in \dom \st :\, (\st-c)[\varphi, \varphi]=0\}.
\]
By Cauchy's inequality one sees that $\ker (\st-c)$
is a linear subspace of $\dom \st$.

\medskip

A basic tool in the present paper is the following concept
of a representing map for forms;
for individual forms see \cite{PW1}, \cite{Szy87},
and for a pair of nonnegative forms see also \cite[(4.6)]{HSS2009}.

\medskip

\begin{definition}%\label{qmap}
Let $\st \in \bF(\sH)$ be a semibounded form
and let $c \leq m(\st)$.
A representing map for the nonnegative form $\st-c$ is a linear operator
$Q \in \bL(\sH, \sK)$, where $\sK$ is a Hilbert space, with $\dom Q=\dom \st$,
such that
\begin{equation}\label{grepq0}
 (\st-c)[\varphi, \psi]=(Q \varphi, Q \psi)_{\sK},
 \quad \varphi, \psi \in \dom Q.
\end{equation}
\end{definition}

In what follows, it is important to observe that
for every semibounded form there exists a representing map.
For a proof of this result, see \cite[Lemma 4.1]{HS2023seq1}.

\begin{lemma}\label{exiuni}
Let $\st \in \bF(\sH)$ be a semibounded form.
Then for each $c \leq m(\st)$ there exists a
representing map $Q \in \bL(\sH,\sK)$ for the nonnegative form $\st-c$.
Moreover, this representing map is uniquely defined in the following sense:
if $Q' \in \bL(\sH, \sK')$ is an another representing map for $\st-c$
with $\dom Q'=\dom \st$,
then there exists a partial isometry $V \in \bB(\sK, \sK')$
with initial space $\cran Q$ and final space $\cran Q'$,
such that $Q'=VQ$.
\end{lemma}

\begin{definition}%\label{rem1}
The representing map $Q \in \bL(\sH, \sK)$
in Lemma {\rm \ref{exiuni}} is called minimal when $\cran Q=\sK$.
\end{definition}

If $Q$ is a representing map for $\st-c$ from $\sH$ to $\sK$,
then by replacing $\sK$ by $\cran Q$ one gets
a representing map for $\st-c$ whose range is dense.
When the representing maps $Q$ and $Q'$ for $\st-c$
both are taken to be operators with dense ranges,
the partial isometry in Lemma \ref{exiuni} becomes a unitary operator.

\medskip

At this point it should be stressed that representing maps $Q$
in \eqref{grepq0} depend on the choice of the bound $c \leq m(\st)$.
By indicating this in the notation $Q_c$ one has, in particular,
\[
 \ker Q_c=\ker (\st-c).
\]
Thus, for instance, $\ker Q_c=\{0\}$ for all $c<m(\st)$,
while in general $\ker Q_c \neq \{0\}$ when $c=m(\st)$.
However, observe that, by definition,
\[
 \dom Q_c=\dom Q_{m(\st)}=\dom \st, \quad c\leq m(\st).
\]
Furthermore, it follows from \eqref{grepq0} that
\begin{equation}\label{brandnew}
\inf \left\{ \,\frac{\|Q_c \varphi\|^2 }{ \|\varphi\|^2} :\,
\varphi \in \dom \st \, \right\}=m(\st)-c.
\end{equation}
Of course, one could always take $c=m(\st)$,
but the more general choice $c \leq m(\st)$ allows
some useful flexibility.

\medskip

The following notions and definitions are quite standard
for a semibounded form $\st \in \bF(\sH)$; see, e.g., \cite{Kato}, \cite{Kosh}.
A sequence $\varphi_n$ in $\dom \st$ is said
to converge to $\varphi \in \sH$ in the sense of $\st$, denoted by
$\varphi_n \to_{\st} \varphi$, if
\[
 \varphi_n \to \varphi \quad \mbox{and} \quad
 \st[\varphi_n-\varphi_m] \to 0,
\]
where the  convergence $\varphi_n \to \varphi$
takes place in $\sH$.

\begin{definition}%\label{einz}
Let $\st \in \bF(\sH)$ be a semibounded form.
Then $\st$  is said to be
\begin{enumerate} \def\labelenumi{\rm(\alph{enumi})}
\item  \textit{ closable} if for any sequence $\varphi_n \in \dom \st$:
\[%begin{equation*}\label{nul}
\hspace{1cm} \varphi_n \to_\st 0 \quad
\Rightarrow \quad \st[\varphi_n] \to 0.
\]%end{equation*}
\item   \textit{ closed} if for any sequence $\varphi_n \in \dom \st$:
\[%begin{equation*}\label{nul}
\hspace{1cm} \varphi_n \to_\st \varphi \quad
\Rightarrow \quad \varphi \in \dom \st \quad \mbox{and}
\quad \st[\varphi_n-\varphi] \to 0.
\]%end{equation*}
 \end{enumerate}
\end{definition}

Closability and closedness of a semibounded form $\st$
are \textit{shift invariant}, i.e.,
for all $a \in \dR$ the forms $\st$ and $\st-a$ are
simultaneously closable or closed, respectively.
The connection between these notions for forms and for the corresponding
linear representing maps is now easily established;
cf. \cite[Ch. VI, Examples 1.13, 1.23]{Kato}.

\begin{lemma}\label{repform0}
Let $\st \in \bF(\sH)$ be a semibounded form
and let $Q_c \in \bL(\sH, \sK)$ be a representing map for $\st$
as in \eqref{grepq0} with $c \leq m(\st)$.
Then the following statements hold:
\begin{enumerate} \def\labelenumi{\rm(\alph{enumi})}
\item $\st$ is  closable if and only if $Q_c$ is closable;
\item $\st$ is  closed if and only if $Q_c$ is closed.
 \end{enumerate}
\end{lemma}

\begin{proof}
It suffices to observe that
$\varphi_n \to_\st \varphi$  is equivalent to
$\varphi_n \to \varphi$
while $Q_c \varphi_n$   converges in $\sK_c$.
\end{proof}

Assume that the semibounded form $\st \in \bF(\sH)$
has a semibounded form extension $\st' \in \bF(\sH)$, i.e., $\st \subset \st'$.
Let $Q_c$ and $Q_c'$ be the representing maps for $\st$
and $\st'$ for $c \leq m(\st)$.
If $\st'$ is closed, then $Q_c'$ is closed by Lemma \ref{repform0},
which implies that $Q_c$ is closable;
thus $\st$ is closable by Lemma \ref{repform0}.
Therefore, if the form $\st$ has a closed extension, then $\st$ is closable.
A converse to this statement is contained in the following lemma.

\begin{lemma}\label{closure}
Let $\st \in \bF(\sH)$ be a semibounded form.
 If $\st$ is closable, then $\st$ has a closure
$\bar \st \in \bF(\sH)$ {\rm (}the smallest closed  extension{\rm )} defined as follows:
the domain $\dom \bar \st$ is given by all $\varphi \in \sH$
for which there exists a sequence
$\varphi_n \in \dom \st$ with  $\varphi_n \to_\st \varphi$, and in this case
\begin{equation}\label{rep1a}
\bar \st[\varphi, \psi]=\lim_{n \to \infty} \st[ \varphi_n, \psi_n]
\quad \mbox{for any} \quad
\varphi_n \to_\st \varphi, \quad \psi_n \to_\st \psi.
\end{equation}
The forms $\st$ and $\bar{\st}$ have the same lower bound.

In fact, if  $c \leq m(\st)$ and $Q_c \in \bL(\sH,\sK)$ is a closable representing map
for $\st-c$:
\[
  \st[\varphi, \psi]-c(\varphi, \psi)= (Q_c \varphi, Q_c  \psi)_\sK,
 \quad \varphi, \psi \in \dom \st =\dom Q_c,
\]
 then $\bar \st$ is given by
\begin{equation}\label{rep1}
 \bar \st[\varphi, \psi] -c(\varphi, \psi)
 = (Q_c^{**} \varphi, Q_c^{**} \psi)_\sK,
 \quad \varphi, \psi \in \dom \bar \st =\dom Q_c^{**},
\end{equation}
where $Q_c^{**}$ denotes the closure of $Q_c$.
\end{lemma}

\begin{proof}
Since $\st$ is closable, also $Q_c$ is closable by Lemma \ref{repform0}, i.e.,
$Q_c^{**}$ is a closed operator.
Hence, the form $\bar{\st}$ defined by \eqref{rep1} is closed
and it clearly extends $\st$.
Since $Q_c^{**}$ is the smallest closed linear
operator extension of $Q_c$,
it follows  that $\bar \st$ is the smallest closed form-extension of $\st$.
By construction one has $\dom \bar{\st}=\dom Q_c^{**}$,
so that $\dom \bar \st$ is given by all $\varphi \in \sH$
for which there exists a sequence $\varphi_n \in \dom \st$
with  $\varphi_n \to_\st \varphi$ and \eqref{rep1a}.
 \end{proof}

If the form $\st \in \bF(\sH)$ is semibounded and closable,
then for any $a \in \dR$ also the form $\st-a \in \bF(\sH)$ is
semibounded and closable. Moreover, in this case one has
\[%begin{equation}\label{closu}
 \overline{\st-a}=\bar{\st}-a, \quad a\in \dR;
\]%end{equation}
cf. Lemma \ref{closure}.
It should be stressed that the closure process in Lemma \ref{closure}
depends neither on the choices of the sequences
$\varphi_n \to_\st \varphi$ and $\psi_n \to_\st \psi$
in \eqref{expl}, \eqref{rep1a}, nor on the choices of
$c \leq m(\st)$ and of $Q_c$ in the representation \eqref{grepq0}.
The notion of a core is closely connected to Lemma \ref{closure}.
Let the form $\st \in \bF(\sH)$ be semibounded and closed.
Then a \textit{core for} $\st$ is a subset $\sD$  of $\dom \st$ such that
the (form) closure of the restriction $\st \uphar \sD$ is equal to $\st$;
cf. \cite[p. 317]{Kato}.
Likewise, let $T \in \bL(\sH, \sK)$ be a closed linear operator.
Then a \textit{core for} $T$  is a subset $\sD$ of $\dom T$ such that
the (graph) closure of $\{ \{\varphi, T \varphi\}:\, \varphi \in \sD\}$ is equal to $T$;
cf. \cite[p. 166]{Kato}.
Therefore, if the closed semibounded form $\st$ with lower bound $\gamma \in \dR$
and the closed linear operator $Q_c$ with $c \leq m(\st)$ are connected by
\[
 \st[\varphi, \psi]-c(\varphi, \psi)=(Q_c\varphi, Q_c \psi),
 \quad \varphi, \psi \in \dom \st=\dom Q_c,
\]
and  $\sD$ is a subset of $\dom \st=\dom Q_c$, then
$\sD$ is a core for $\st$ if and only if $\sD$ is a core for $Q_c$;
cf. Lemma \ref{closure}.

\medskip

The translation invariance of the closability of forms $\st \in \bF(\sH)$ is quite natural.
The situation is different in the case of nonnegative singular forms.
For instance, it will follow from the definition below that
if $\st \in \bF(\sH)$ is nonnegative and singular,
then $\st+c$ is not singular for any $c >0$.

\begin{definition}%\label{singform}
Let the form $\st \in \bF(\sH)$ be nonnegative.
Then $\st$ is said to be singular
if for every $\varphi \in \dom \st$
there is a sequence $\varphi_n \in \dom \st$ such that
\begin{equation}\label{nula2}
 \varphi_n\to \varphi,\quad \st[\varphi_n] \to 0.
\end{equation}
\end{definition}

Clearly, an equivalent definition is that for every $\varphi \in \dom \st$
there is a sequence $\varphi_n \in \dom \st$ such that
\begin{equation}\label{nula22}
 \varphi_n\to 0,\quad \st[\varphi_n-\varphi] \to 0.
\end{equation}
For a discussion of singular nonnegative forms, see for instance \cite{Kosh}.
In \cite{S3} the singular part of a form was defined via the Lebesgue
decomposition as the difference between the form and its regular part;
see Section \ref{sec31}.

\begin{lemma}\label{repform01}
Let the form $\st \in \bF(\sH)$ be nonnegative
and let $Q \in \bL(\sH, \sK)$ be a representing map for $\st$, so that
\begin{equation}\label{grepq0+}
 \st[\varphi, \psi]=(Q \varphi, Q \psi)_{\sK},
 \quad \varphi, \psi \in \dom Q,
\end{equation}
and, consequently,
\begin{equation}\label{mq}
m(\st)=\inf \left\{ \,\frac{\|Q \varphi\|^2 }{ \|\varphi\|^2} :\,
\varphi \in \dom Q \, \right\}.
\end{equation}
Then $\st$ is  singular if and only if $Q$ is singular, and in this case
$m(\st)=0$.
\end{lemma}

\begin{proof}
Since $\st$ is nonnegative it follows that $m(\st) \geq 0$.
Thus by taking $c=0$ in \eqref{grepq0} one sees that \eqref{grepq0+} holds
for some representing map $Q$ from $\sH$ to a Hilbert space $\sK$.
Furthermore, now \eqref{mq} is a consequence of \eqref{brandnew}.

According to the definition in \eqref{nula2} the form $\st$ is singular
if and only if for every $\varphi\in \dom \st=\dom Q$
there is a sequence $\varphi_n\in \dom \st$ such that
\begin{equation}\label{duf}
  \varphi_n\to \varphi,\quad \st[\varphi_n]=\|Q\varphi_n\|^2 \to 0.
\end{equation}
This condition means that $\dom Q\subset \ker Q^{**}$,
i.e., $Q$ is a singular operator. Moreover, it follows from \eqref{duf}
and \eqref{mq} that $m(\st)=0$. Conversely, if $Q$ is singular, it follows
from \eqref{grepq0+} that the form $\st$ is nonnegative and singular.
\end{proof}

It is clear that a null form $\st \in \bF(\sH)$, i.e., $\dom \st=\ker \st$,
is nonnegative and simultaneously, closable and singular.
Conversely, if $\st \in \bF(\sH)$ is a nonnegative form
which is both regular and singular, then it follows
from Lemma \ref{repform01} that $\st=0$.

\medskip

Recall that for linear operators $Q_1 \in \bL(\sH, \sK_1)$ and $Q_2 \in \bL(\sH, \sK_2)$
%from $\sH$ to  Hilbert spaces $\sK_1$ and $\sK_2$, respectively,
one says that $Q_1$ is \textit{contractively dominated} by $Q_2$,
in notation $Q_1 \prec_c Q_2$, if
 \begin{equation}\label{prec}
\dom Q_2 \subset \dom Q_1 \quad \mbox{and} \quad
\|Q_1 \varphi \| \leq \|Q_2 \varphi \|,  \quad \varphi \in \dom Q_2,
\end{equation}
see \cite[Definition 8.1, Lemma 8.2]{HSS2018}.
If for $i=1,2$ one has
\[
\st_i[\varphi, \psi]=c(\varphi, \psi)+ (Q_i \varphi, Q_i \psi)_{\sK_i},
 \quad \varphi, \psi \in \dom \bar \st_i =\dom Q_i
\]
then clearly
\begin{equation}\label{prec1}
 \st_1 \leq \st_2 \quad \Leftrightarrow \quad Q_1 \prec_c Q_2.
\end{equation}
This equivalence will be tacitly used elsewhere in the paper.
In particular, one sees that if $\st_1$ and $\st_2$ are closable, then
\[%begin{equation}\label{prec1+}
 \st_1 \leq \st_2 \quad \Rightarrow \quad \bar{\st}_1 \leq \bar{\st}_2,
\]%end{equation}
as  $Q_1 \prec_c Q_2$ implies that $(Q_1)^{**} \prec_c (Q_2)^{**}$;
cf. \cite{HSS2018}, \cite[(2.4)]{HS2023seq1}.

\medskip

Finally some remarks will be made concerning the definitions
of closable, closed, and singular forms in a Hilbert space.

\begin{remark}%\label{clofor}
Assume that the form $\st \in \bF(\sH)$ is semibounded
and let $Q_c \in \bL(\sH, \sK)$ be a representing map,
with $c \leq \gamma$.
If $\varphi_n \to_\st \varphi$ and $\psi_n \to_\st \psi$,
 then $\varphi_n \to \varphi$ and $\psi_n \to \psi$ in $\sH$,
while $Q_c \varphi_n$ and $Q_c \psi_n$ converge
in $\sK_c$ to, say, $h$ and $k$ in $\ran Q_c^{**}$.
This implies that
\begin{equation}\label{expl}
\lim_{n \to \infty} \st[\varphi_n, \psi_n]
=c (\varphi, \psi)+  (h,k).
\end{equation}
 Assume that $\mul Q_c^{**}$ is not trivial.
Let $\omega \in \mul Q_{c}^{**}$ with $\omega \neq 0$.
Then there is a sequence $\chi_n\in \dom Q_c$
such that $\chi_n\to 0$ and $Q_c \chi_n\to \omega$.
Therefore $\varphi_n+ a \chi_n \to_\st \varphi$
for all $a\in\dC$ and therefore
\[
 \lim_{n \to \infty} \st[\varphi_n+a \chi_n,\psi_n]=c(\varphi,\psi)+(h,k)+a(\omega,k).
\]
Thus it follows that for any $\omega \in \mul Q_c^{**}$ which is not orthogonal
to $k$ the limit in \eqref{expl} can take every possible value in $\dC$, as can be seen from
 \[
 \ran Q_c^{**} =( \ran Q_c^{**} \ominus \mul Q_c^{**} ) \oplus \mul Q_c^{**}.
\]
Consequently, the limit in \eqref{expl} depends on the choice of the sequences
$\varphi_n$ and $\psi_n$.
Hence, in general, there does not exist a well defined extension for the form $\st$.

Finally, notice that the alternative characterization of singularity in \eqref{nula22}, i.e.
for every $\chi\in \dom \st$ there is a sequence $\chi_n \in \dom \st$, such that
\[%begin{equation}\label{nula2a}
 \chi_n\to 0,\quad \mbox{and} \quad \st[\chi_n-\chi] \to 0,
\]%end{equation}
is equivalent to saying that $\ran Q \subset \mul Q^{**}$.
\end{remark}

\section{Lebesgue  decompositions of semibounded forms}\label{sec31}

It will be shown that for a semibounded form $\st \in \bF(\sH)$ there exists a sum decomposition
$\st=\st_1+\st_2$, $\dom \st=\dom \st_1=\dom \st_2$, where $\st_1 \in \bF(\sH)$ is closable
and $\st_2 \in \bF(\sH)$ is singular. Moreover, it will be shown that this decomposition is unique
by the property that $\st_1$ is the largest of all closable forms below $\st$.
This so-called Lebesgue decomposition and its uniqueness go back to Simon \cite{S3};
the present proofs are based on the Lebesgue decomposition of the representing map of
the form $\st$.

\medskip

Let the form $\st \in \bF(\sH)$ be semibounded,
let $c \leq m(\st)$, and assume that $\st$ has the representation
\begin{equation}\label{LlebB}
\st[\varphi, \psi]=c(\varphi, \psi)+(Q \varphi, Q \psi), \quad \varphi, \psi \in \dom \st=\dom Q,
\end{equation}
where $Q \in \bL(\sH, \sK)$ is a representing map for $\st-c$.
The closure of $Q$ is a closed linear relation $Q^{**} \in \bL(\sH,\sK)$ and
let $P_0$ be the orthogonal projection from $\sK$ onto $\mul Q^{**}$.
Then the sum decomposition
\begin{equation}\label{lebtyp+2}
 Q=Q_{\rm reg}+Q_{\rm sing}, \quad Q_{\rm reg}=(I-P_0)Q, \quad   Q_{\rm sing}=P_0Q, \\
\end{equation}
where
$\dom Q_{\rm reg}=\dom Q_{\rm sing}=\dom Q$,
is the so-called Lebesgue decomposition of the representing $Q$; cf. \cite{HSS2018}, \cite{HS2023a}.
Here $Q_{\rm reg}$ is a closable operator
and $Q_{\rm sing}$ is a singular operator.
 Recall that $Q^{**} \in \bL(\sH, \sK)$ has a similar decomposition
\[%begin{equation}\label{lebtyp+22}
 Q^{**}=(Q^{**})_{\rm reg}+(Q^{**})_{\rm sing}, \quad
 (Q^{**})_{\rm reg}=(I-P_0)Q^{**}, \quad   (Q^{**})_{\rm sing}=P_0Q^{**}, \\
\]%end{equation}
where
$\dom (Q^{**})_{\rm reg}=\dom (Q^{**})_{\rm sing}=\dom Q^{**}$.
Here $(Q^{**})_{\rm reg}$ is the usual orthogonal operator part
of the closed linear relation $Q^{**}$.
It is important to observe that
\[%begin{equation}\label{qreg}
(Q_{\rm reg})^{**}=(Q^{**})_{\rm reg}, \quad \dom (Q_{\rm reg})^{**}=\dom Q^{**}.
\]%end{equation}
For more details, see  \cite{HSS2018}, \cite{HS2023a}, \cite{HS2023seq1}.

\begin{theorem}\label{tleb}
Let the form $\st \in \bF(\sH)$ be semibounded.
Then $\st$ has the sum decomposition
\begin{equation}\label{lebtyp+}
 \st=\st_{\rm reg}+\st_{\rm sing},
 \quad \dom \st=\dom \st_{\rm reg}=\dom \st_{\rm sing},
\end{equation}
with a closable semibounded form $\st_{\rm reg} \in \bF(\sH)$
and a singular nonnegative form $\st_{\rm sing} \in \bF(\sH)$.
In fact, let $c \leq m(\st)$ and assume
that $\st$ has the representation \eqref{LlebB}.
Then
\begin{equation}\label{lebtyp+3}
\st_{\rm reg}[\varphi, \psi]
=c (\varphi, \psi) +(Q_{{\rm reg}} \varphi, Q_{{\rm reg}} \psi),
\quad \st_{\rm sing}[\varphi,
\psi]=(Q_{{\rm sing}} \varphi, Q_{ {\rm sing}} \psi),
\end{equation}
with $\varphi, \psi \in \dom \st=\dom Q$.
Moreover, the closure $\overline{{\st}_{\rm reg}}$
of the closable part $\st_{\rm reg}$ is given by
\begin{equation}\label{lebtyp+4}
\overline{{\st}_{\rm reg}}\,[\varphi, \psi]
=c (\varphi, \psi) +((Q_{{\rm reg}})^{**} \varphi, (Q_{{\rm reg}})^{**}  \psi),
\quad  \varphi, \psi \in \dom \overline{\st_{\rm reg}}=  \dom Q^{**}.
\end{equation}
\end{theorem}

\begin{proof}
Since $P_0$ in \eqref{lebtyp+2} is an orthogonal projection, it  is clear from \eqref{lebtyp+3}
that $\st=\st_{\rm reg}+\st_{\rm sing}$; cf. \eqref{LlebB}.
Since $Q_{\rm reg}$ is a closable operator and $Q_{\rm sing}$ is a singular operator
it also follows from \eqref{lebtyp+3} that the forms $\st_{\rm reg}$ and $\st_{\rm sing}$
are regular and singular, respectively; cf. Lemma \ref{repform0} and Lemma \ref{repform01}.
Finally, \eqref{lebtyp+4} is a direct consequence of the definition of $\st_{\rm reg}$;
cf. Lemma \ref{closure}.
\end{proof}

The decomposition of the semibounded form $\st \in \bF(\sH)$ in \eqref{lebtyp+}
into the regular part $\st_{\rm reg}$ and the singular part $\st_{\rm sing}$
in \eqref{lebtyp+3} is called the \textit{Lebesgue decomposition} of $\st$.
It will be shown in the following theorem that $\st_{\rm reg}$ in
\eqref{lebtyp+3} does not depend on the choice of $c$ and $Q$ in \eqref{LlebB},
and neither does $\st_{\rm sing}=\st-\st_{\rm reg}$.

\begin{theorem}\label{tleb1}
Let the form $\st \in \bF(\sH)$ be semibounded.
Then $\st_{\rm reg}$ as defined in Theorem {\rm{\ref{tleb}}}
satisfies $\st_{\rm reg} \leq \st$.
Let $\su \in \bF(\sH)$ be a semibounded form and
assume that $\su \leq \st$, then
\begin{equation}\label{lebtyp++}
 \su_{\rm reg} \leq \st_{\rm reg}.
\end{equation}
In particular, if the form $\su \in \bF(\sH)$ is semibounded
and $\su \leq \st$, then
\[
 \su \leq \st_{\rm reg}.
\]
Consequently,  $\st_{\rm reg}$ is the largest
of all closable semibounded forms below $\st$.
\end{theorem}

\begin{proof}
By assumption $\su \in \bF(\sH)$ is a semibounded form with $\su \leq \st$, i.e.,
\[
 \dom \st \subset \dom \su \quad \mbox{and}
 \quad \su[\varphi] \leq \st[\varphi], \quad \varphi \in \dom \st.
\]
Let the lower bound of $\su$ be $c \in \dR$.
Thanks to the inequality $\su \leq \st$, it follows that $c \leq m(\st)$.
By Lemma \ref{exiuni}  the form $\su$ (having lower bound $c$)
has a representation of the form
\[
 \su[\varphi, \psi]=c (\varphi, \psi)+
 (R_c \varphi, R_c \psi), \quad \varphi, \psi \in \dom \su=\dom R_c,
\]
with a representing map $R_c$ from $\sH$ to a Hilbert space $\sK_c'$.
Since $c \leq m(\st)$, it follows from the same lemma that the form $\st$
(with lower bound $\gamma$) has a representation of the form
\[%begin{equation}\label{lebtyp+11}
\st[\varphi, \psi]=c  (\varphi, \psi) +(Q_c \varphi, Q_c  \psi),
\quad \varphi, \psi \in \dom \st=\dom Q_c,
\]%end{equation}
where $Q_c$ is a representing map from $\sH$ to a Hilbert space $\sK_c$.
As in Theorem \ref{tleb} one sees that
\[
\su_{\rm reg}[\varphi, \psi]
=c (\varphi, \psi)+(R_{c, {\rm reg}} \varphi, R_{c,{\rm reg}} \psi),
\quad \varphi, \psi \in \dom R_c,
\]
and, likewise, it follows from Theorem \ref{tleb} that
\[
\st_{\rm reg}[\varphi, \psi]=c (\varphi, \psi)+(Q_{c, {\rm reg}} \varphi, Q_{c, {\rm reg}} \psi),
\quad \varphi, \psi \in \dom Q_c.
\]
Using the representations of $\su$ and $\st$ one sees that the inequality $\su \leq \st$
is equivalent to
\[
 \dom Q_c \subset \dom R_c \quad \mbox{and} \quad
  \|R_c \varphi\| \leq  \|Q_c \varphi\|,
 \quad \varphi \in \dom Q_c.
\]
As a consequence from \cite[Theorem 8.3]{HSS2018}, one finds that
\begin{equation}\label{presreg}
 R_{c, {\rm reg}} \prec Q_{c, {\rm reg}},
\end{equation}
see \eqref{prec}.
Now, using the above representations of $\su_{\rm reg} $ and $\st_{\rm reg}$,
it follows that the inequality \eqref{lebtyp++} holds; cf. \eqref{prec1}.
Since $\st_{\rm reg} \leq \st$ and $\st_{\rm reg}$ is closable, it is clear that
$\st_{\rm reg}$ has the stated maximality property.
\end{proof}

\begin{remark}%\label{pertLem}
The following useful observation goes back to \cite{S3}.
Let the form $\st \in \bF(\sH)$ be semibounded,
let the form $\sb\in \bF(\sH)$  be symmetric
with $\dom \st\subset \dom \sb$, and assume that for some $\alpha \geq 0$
\[
|\sb[\varphi]| \leq \alpha \|\varphi\|^2,
\quad \varphi \in \dom \st.
\]
Then the sum $\st+\sb \in \bF(\sH)$
is semibounded (with $\dom (\st+\sb)=\dom \st$).
Moreover, if $\st$ is closable, then $\st+\sb$ is closable,
cf. \cite[p. 320]{Kato} and \cite[Theorem 5.1.16]{BHS},
and in this case
\begin{equation}\label{regsingpert}
  (\st+\sb)_{\rm reg} =  \st_{\rm reg} +\sb
  \quad \mbox{and} \quad
   (\st+\sb)_{\rm sing} =  \st_{\rm sing}.
\end{equation}
To see this, note that  $\st_{\rm reg} \leq \st$ implies
$\st_{\rm reg} +\sb \leq \st+\sb$,
and, since the left-hand side is closable, one obtains
$\st_{\rm reg} +\sb \leq (\st+\sb)_{\rm reg}$ by Theorem \ref{tleb1}.
Likewise $(\st+\sb)_{\rm reg} -\sb$ is closable
and $(\st+\sb)_{\rm reg} -\sb \leq \st$, so that
$(\st+\sb)_{\rm reg} -\sb \leq \st_{\rm reg}$ by Theorem \ref{tleb1}.
Combining these inequalities gives the first identity in \eqref{regsingpert}
 and the second identity in \eqref{regsingpert} is clear from the first one.
\end{remark}

It follows from \eqref{lebtyp+2}  that
$\ran Q_{\rm reg} \subset \ran (I-P_0)$ and $\ran Q_{\rm sing} \subset \ran P_0$.
If the representing map $Q$ is minimal: $\cran Q=\sK$, then
$\cran Q_{\rm reg} = \ran (I-P_0)$ and $\cran Q_{\rm sing} = \ran P_0$.
Clearly, the regular and singular  components in the Lebesgue decomposition
of  $Q$ have only a trivial intersection of the ranges.
The Lebesgue decomposition is an example of a more general decomposition
of the form $\st=\st_1+\st_2$ where $\st_1$ is closable and $\st_2$ is singular.
In fact, by considering orthogonal Lebesgue type decompositions of the
representing map $Q$ one may easily construct such decompositions in the same way
as the Lebesgue decomposition.  Moreover, also in these cases the ranges
of the components of $Q$ have only a trivial intersection; cf. Corollary \ref{Lebtypepp}.
One then says that the corresponding forms $\st_1$ and $\st_2$
are mutually singular.
However not every decomposition $\st=\st_1+\st_2$ where $\st_1$ is closable and
$\st_2$ is singular, can be obtained via an orthogonal Lebesgue type decomposition of the
representing map $Q$.
The required class of decompositions is much wider  and
only in this more general situation one meets components $\st_1$ and $\st_2$
that are not mutually singular; cf. Theorem \ref{Lebtypep}.
In order to characterize all such decompositions it is helpful to first
investigate the sum decompositions of nonnegative forms into nonnegative forms
in Section \ref{sec3}. These will then be used to study Lebesgue type decompositions
of semibounded forms in Section \ref{sec32}.
%For the convenience of the reader
%Section \ref{app} contains a review of a number of properties of products of the form
%$Q^*Q$ in terms of the regular part $Q_{\rm reg}$.

\section{Sum decompositions of nonnegative forms}\label{sec3}

In this section the interest will be in the sum decomposition $\sh=\sh_1+\sh_2$
of a nonnegative form $\sh \in \bF(\sH)$ with nonnegative forms
$\sh_1, \sh_2 \in \bF(\sH)$
with the additional property that $\dom \sh=\dom \sh_1=\dom \sh_2$.
The parametrization of such decompositions will be given by means of the
representing map of the sum.
The minimality of the representation of such a sum will be characterized.
Furthermore, the interaction of the components $\sh_1$ and $\sh_2$
in the sum decomposition will be described in terms of
the parallel sum of $\sh_1 : \sh_2$.

\medskip

First a characterization will be presented of all possible sum decompositions
of a nonnegative form. The construction in Theorem \ref{Lebtype} is based on the fact that
associated with a sum decomposition of nonnegative forms are natural inequalities
that arise from this decomposition; cf. \cite{PW1}.

\begin{theorem}\label{Lebtype}
Let $\sh \in \bF(\sH)$ be a nonnegative form
and assume that $\sh$ has the representation
\begin{equation}\label{lebB}
\sh[\varphi, \psi]=
(Q \varphi, Q \psi), \quad \varphi, \psi \in \dom \sh=\dom Q,
\end{equation}
where $Q \in \bL(\sH, \sK)$ is a representing map.

Let $K\in \bB(\sK)$ be a nonnegative contraction
and define the forms $\sh_1, \sh_2 \in \bF(\sH)$
 for all $\varphi, \psi \in \dom \sh=\dom Q$ by
\begin{equation}\label{lebB2}
 \sh_1[\varphi, \psi]=
 ((I-K)^{\half}Q \varphi, (I-K)^{\half}Q \psi), \quad
 \sh_{2}[\varphi, \psi]=(K^{\half}Q \varphi, K^{\half}Q \psi).
\end{equation}
Then the form $\sh$ has the sum decomposition
 \begin{equation}\label{lebB1}
 \sh=\sh_1+\sh_2, \quad \dom \sh=\dom \sh_1=\dom \sh_2,
\end{equation}
where $\sh_1$ and $\sh_2$ are nonnegative.

Conversely, let the nonnegative form $\sh \in \bF(\sH)$ in \eqref{lebB}
have a sum decomposition \eqref{lebB1},
where $\sh_1, \sh_2 \in \bF(\sH)$ are nonnegative.
Then there exists a nonnegative contraction $K\in \bB(\sK)$
such that  \eqref{lebB2} holds.
\end{theorem}

\begin{proof}
Assume that $K \in \bB(\sK)$ is a nonnegative contraction
and let $\sh_1, \sh_2 \in \bF(\sH)$ be defined by \eqref{lebB2}.
Due to the identity
\[
 ((I-K)^{\half})^*(I-K)^{\half}+(K^{\half})^*K^{\half}=K+(I-K)=I_\sK,
\]
it is clear from \eqref{lebB2}
that $\sh_1+\sh_2=\sh$. Hence, $K$ generates a sum decomposition
of the form $\sh$ in \eqref{lebB}.

\medskip

Conversely, let the form $\sh$ have the decomposition
as in \eqref{lebB1}.
According to Lemma \ref{exiuni} there exist representing maps
$Q_1$ and $Q_2$ from $\sH$ to Hilbert spaces
$\sK_1$ and $\sK_2$, respectively, with domains equal to $\dom \sh$,
such that
\begin{equation}\label{ipsum0}
\sh_1[\varphi, \psi] =
(Q_1 \varphi, Q_1 \psi)_{\sK_{1}}, \quad
\sh_2[\varphi, \psi] = (Q_2 \varphi, Q_2 \psi)_{\sK_{2}},
\end{equation}
for all $\varphi, \psi \in \dom \sh$.
It follows from \eqref{lebB1} and \eqref{ipsum0} that
\begin{equation}\label{ipsum}
\begin{split}
(Q \varphi, Q \psi)_{\sK}
 &=\sh[\varphi, \psi] = \sh_1[\varphi, \psi]+ \sh_2[\varphi, \psi] \\
& =(Q_{1}\varphi, Q_{1}\psi)_{\sK_1} + (Q_{2}\varphi, Q_{2}\psi)_{\sK_2},
 \end{split}
\end{equation}
for all $\varphi, \psi \in \dom \sh$.
The identity \eqref{ipsum} shows that there are natural
inequalities: for all $\varphi \in \dom \sh$ one has
\[
 (Q_{1}\varphi , Q_{1}\varphi)_{\sK_1}
 \leq \big(Q \varphi , Q \varphi \big)_{\sK}, \quad
  (Q_{2}\varphi , Q_{2}\varphi )_{\sK_2}
 \leq \big( Q \varphi , Q \varphi \big)_{\sK}.
\]
Hence, the linear relations $C_1$ from $\sK$ to $\sK_{1}$
and $C_2$ from $\sK$ to $\sK_{2}$, defined by
\[%begin{equation}\label{jot}
 C_1 =\big\{ \{Q \varphi , Q_{1}\varphi \} :\, \varphi \in \dom \sh \big\},
\quad
 C_2 =\big\{ \{Q \varphi, Q_{2}\varphi \} :\, \varphi \in \dom \sh \big\},
\]%end{equation}
are actually contractive operators from $\ran Q$ onto $\ran Q_1$
and $\ran Q_2$, respectively. Thus they can be uniquely extended to
all of $\cran Q$ and these extensions are denoted again by $C_1$ and $C_2$,
respectively. Finally extend these operators trivially to all of $\sK$
and give these extensions the same notation.
Now define the operators $K_1$ and $K_2$ by
\[
 K_1=C_1^*C_1^{**} \quad \mbox{and} \quad K_2=C_2^*C_2^{**}.
\]
Then $K_1,K_2\in\bB(\sK)$ are nonnegative contractions and they satisfy
\begin{equation}\label{Jot1}
 \sh_1[\varphi, \psi]
  = \big(  C_1 Q \varphi , C_1 Q \psi  \big)_{\sK_{1}}
   =\big( K_1  Q \varphi ,  Q \psi  \big)_{\sK}
\end{equation}
and
\begin{equation}\label{Jot2}
 \sh_2[\varphi, \psi]
  = \big(  C_2 Q \varphi , C_2 Q \psi  \big)_{\sK_{2}}
   =\big( K_2  Q \varphi ,  Q \psi  \big)_{\sK}
\end{equation}
for all $\varphi, \psi  \in \dom \sh$.
By combining \eqref{Jot1} and \eqref{Jot2} with \eqref{ipsum},
it follows that
\[%begin{equation}\label{Jot3}
 K_1 + K_2 = P_{\cran Q},
\]%end{equation}
where the right-hand side stands for the orthogonal projection from $\sK$ onto $\cran Q$.
Now let $K=K_2$, then the second identity in \eqref{lebB2} follows from \eqref{Jot2}.
Note that
\[
 I_\sK-K=P_{\cran Q} -K_2  +(I-P_{\cran Q})=K_1+(I-P_{\cran Q}),
\]
which leads to the identity
\[
(( I_\sK-K)Q\varphi, Q\psi)= (K_1 Q \varphi, Q \psi)_{\sK}, \quad \varphi, \psi \in \dom \sh.
\]
Thus the first identity in \eqref{lebB2} follows from \eqref{Jot1}.
\end{proof}

\begin{remark}
If the representing map $Q \in \bL(\sH, \sK)$ in \eqref{lebB} is minimal, %i.e., $\cran Q=\sK$,
then the contraction $K \in \bB(\sK)$ in \eqref{lebB2}  is unique
(see for a similar feature \cite[Remark 3.7]{HS2023a}).
To see this directly, assume that $K$ and $\wt K$ are nonnegative contractions
in $\bB(\sK)$ which satisfy \eqref{lebB2}.
Then it follows that  ,
\[
 \sh_{2}[\varphi, \psi]=(K^{\half}Q \varphi, K^{\half}Q \psi)
 =(\wt K^{\half}Q \varphi, \wt K^{\half}Q \psi), \quad \varphi, \psi\in\dom \sh.
\]
Hence $((K-\wt K)Q \varphi, Q \psi)=0$ for all $\varphi, \psi\in \dom \sh$.
By the assumption $\cran Q=\sK$ one concludes that $K-\wt K=0$.
\end{remark}

Let $Q_1 \in \bL(\sH, \sK_1)$ and $Q_2 \in \bL(\sH, \sK_2)$ be linear operators;
here $\sK_1$ and $\sK_2$ are Hilbert spaces.
The corresponding \textit{column operator} ${\rm col\,} (Q_1,Q_2)$
from $\sH$ to the orthogonal sum $\sK_{1} \oplus \sK_2$  is defined by
\[%begin{equation}\label{colum}
 {\rm col\, } (Q_1,Q_2)=\begin{pmatrix} Q_{1} \\ Q_2 \end{pmatrix}:
 \dom Q_1 \cap \dom Q_2 \to \sK=\begin{pmatrix} \sK_{1} \\ \sK_2 \end{pmatrix}.
\]%end{equation}
The column operator is used to describe the representing map for a sum
of nonnegative forms. Let the forms $\sh_1, \sh_2 \in \bF(\sH)$
be nonnegative with $\dom \sh_1=\dom \sh_2$.
For $i=1,2$ let $Q_i \in \bL(\sH, \sK_i)$ be the representing maps for the forms $\sh_i$
with $\dom \sh=\dom Q_1=\dom Q_2$, such that
\begin{equation}\label{ortho1}
 \sh_1[\varphi, \psi]= (Q_{1} \varphi, Q_{1} \psi),
 \quad
 \sh_2[\varphi, \psi]=(Q_{2} \varphi, Q_{2} \psi),
 \end{equation}
where $\varphi, \psi\in \dom \sh_1=\dom \sh_2=\dom Q_1=\dom Q_2$.
Then ${\rm col\,} (Q_1,Q_2)$ is a representing map for the form $\sh_1+\sh_2$,
as follows with  $\varphi, \psi \in \dom \sh_1=\dom \sh_2$ from
\begin{equation}\label{ortho11}
\begin{split}
  (\sh_1+\sh_2)[\varphi, \psi]
  &= (Q_{1} \varphi, Q_{1} \psi +(Q_2 \varphi, Q_2 \psi) \\
  &= ({\rm col\,} (Q_1,Q_2)  \varphi, {\rm col\,} (Q_1,Q_2)  \psi).
\end{split}
\end{equation}
If in \eqref{ortho1} both $Q_1$ and $Q_2$ are minimal
in $\sK_1$ and $\sK_2$, respectively,
it does not necessarily follow that $\col (Q_1, Q_2)$ is minimal
in $\sK_1 \oplus \sK_2$.
A characterization of minimality in the context of Theorem \ref{Lebtype}
will be given in Theorem \ref{minim}.
 For this purpose let $K \in \bB(\sK)$ be a nonnegative contraction and
define the  Hilbert spaces $\sK_1$ and $\sK_2$ by
\[
\sK_1=\cran (I-K) \quad  \mbox{and} \quad \sK_2=\cran K,
\]
each provided with the inner product of $\sK$.
They satisfy $\sK=\sK_1+\sK_2$ with overlapping space
$\sK_1 \cap \sK_2=\cran (I-K)K$; for a treatment of the corresponding
de Branges-Rovnyak decompositions, see \cite{HS2023a}.
With the representations in \eqref{lebB2}
one defines $Q_1=(I-K)^\half Q$ and $Q_2=K^\half Q$. Then $\col (Q_1, Q_2)$
is a representing map in $\bL(\sH, \sK_1\times\sK_2)$,
so that for all $\varphi, \psi \in \dom \sh=\dom Q$
\begin{equation}\label{lebB3}
\sh[\varphi, \psi]=
\left(
\begin{pmatrix}
     (I-K)^{\half} \\
     K^{\half}
\end{pmatrix}Q \varphi,
\begin{pmatrix}
     (I-K)^{\half} \\
     K^{\half}
\end{pmatrix}Q \psi
\right).
\end{equation}

\begin{theorem}\label{minim}
Let the form $\sh \in \bF(\sH)$ be nonnegative
with the representation \eqref{lebB},
where $Q \in \bL(\sH, \sK)$ is a minimal representing map, i.e., $\cran Q=\sK$.
Let $K \in \bB(\sK)$ be a nonnegative contraction
and assume that with respect to $K$ the form $\sh$ has the representation \eqref{lebB3}.
Then $(I-K)^{\half}Q$ is an operator from $\sH$ into $\sK_1$ with dense range
and $K^{\half}Q$ is an operator from $\sH$ into $\sK_2$ with dense range.
Moreover, the following statements are equivalent:
\begin{enumerate} \def\labelenumi{\rm(\roman{enumi})}
\item
the representation \eqref{lebB3} is minimal;
\item
$K$ is an orthogonal projection.
\end{enumerate}
\end{theorem}

\begin{proof}
The representation \eqref{lebB3} for the form $\sh \in \bF(\sH)$
follows from \eqref{lebB2} and \eqref{ortho11}.
By assumption $Q$ has dense range in $\sK$, which implies that $(I-K)^{\half}Q$
has dense range in $\cran(I-K)^{\half}=\sK_1$
and similarly $K^{\half}Q$ has dense range in $\cran K^{\half}=\sK_2$.

Recall from \cite{HS2023a} that $W=\col ((I-K)^{\half}, K^{\half})$
is an isometry from $\sK$ into the Hilbert space $\sK_1 \oplus \sK_2$.
Moreover, the isometry $W$ is surjective
if and only if $K$ is an orthogonal projection; cf. \cite[Proposition 2.8]{HS2023a}.
Finally observe that due to $\cran Q=\sK$ the mapping $W$ is surjective
if and only if the representing map $\col ((I-K)^{\half}Q, K^{\half}Q)$ is minimal.
\end{proof}

Next the interaction between the components in a sum decomposition
will be considered.  This interaction is measured in terms
of parallel sums.
To explain this, assume that $\sh_1, \sh_2 \in \bF(\sH)$ are nonnegative
with $\dom \sh_1=\dom \sh_2$.
 The \textit{parallel sum} of $\sh_1$ and $\sh_2$ is a nonnegative form
defined for $\varphi \in \dom \sh_1=\dom \sh_2$ by
\begin{equation}\label{parsum}
 (\sh_1 : \sh_2)[\varphi] =\inf
 \left\{ \sh_1[h+\varphi]+\sh_2[h] :\, h \in \dom \sh \right\},
\end{equation}
and polarization.
The definition in \eqref{parsum} and a proof that it is actually
a nonnegative form can be found in \cite[Proposition~2.2]{HSS2009}.
Likewise, for nonnegative operators $A, B \in \bB(\sK)$ one defines
the \textit{parallel sum} $A : B \in \bB(\sK)$ for $\varphi \in \sK$ by
\begin{equation}\label{parsum0}
( (A : B) \varphi, \varphi) = \inf
\left\{ (A(h+\varphi), h +\varphi)+(Bh, h) :\,   h \in \sK  \right\},
\end{equation}
and polarization, cf. \cite{FW}.

\begin{proposition}\label{repparsum}
Let the forms $\sh_1, \sh_2 \in \bF(\sH)$ be nonnegative with sum  $\sh=\sh_1+\sh_2$.
Let $Q \in \bL(\sH,\sK)$  be a minimal representing map for the sum $\sh$:
\[
 \sh[\varphi, \psi]= (Q\varphi, Q \psi),  \quad \varphi, \psi \in \dom \sh.
\]
Let $K\in\mathbf{B}(\sK)$ be the nonnegative contraction for which:
\[
\sh_1[\varphi, \psi]= ((I-K)^\half Q\varphi, (I-K)^\half Q \psi),   \quad
\sh_2[\varphi, \psi]= (K^\half Q\varphi, K^\half Q \psi),
\]
where $\varphi, \psi \in \dom \sh=\dom Q$.
Then the parallel sum $\sh_1:\sh_2 \in \bF(\sH)$ has the representation
\begin{equation}\label{ppaarr}
 (\sh_1:\sh_2)[\varphi,\psi]=\left(( (I-K):K)Q\varphi,Q\psi\right),
 \quad \varphi,\psi\in\dom \sh=\dom Q.
\end{equation}
\end{proposition}

\begin{proof}
Let $\varphi\in\dom Q$ and consider
the quadratic form defined by \eqref{parsum}.
It follows from the definition in \eqref{parsum}
and the representations \eqref{lebB2} in Theorem~\ref{Lebtype} that
\begin{equation}\label{parsum2}
\begin{array}{rl}
   (\sh_1:\sh_2)[\varphi]& =\inf
 \big\{ \|(I-K)^{\half}Q(h+\varphi)\|^2+\|K^{\half}Qh\|^2 :\, h\in \dom Q \big\} \\
 & \geq
 \inf
 \big\{ \|(I-K)^{\half}(Q\varphi-g)\|^2+\|K^{\half}g\|^2 :\,  g\in\sK  \big\} \\
 & = \inf
 \left\{ \left((I-K)f,f\right)+\left(Kg,g\right):\, f+g=Q\varphi :\,  f,g\in\sK \right\} \\
 & = \left(( (I-K):K)Q\varphi,Q\varphi\right).
\end{array}
\end{equation}
The last equality follows from the definition in \eqref{parsum0}.

The reverse inequality in \eqref{parsum2} follows from
the denseness of $\ran Q$ in $\sK$.
To see this let $\varepsilon>0$ and select $g\in\sK$ such that
\[
 \|(I-K)^{\half}(Q\varphi-g)\|^2+\|K^{\half}g\|^2
 < \left(( (I-K):K)Q\varphi,Q\varphi\right)+\varepsilon.
\]
Since $\ran Q$ is dense in $\sK$, one has
$\lim_{n\to \infty}\|g-Qh_n\|=0$ for some sequence $h_n\in \dom Q$.
It follows that for every $\varepsilon>0$ there exists $h\in \dom Q$ such that
\[
 \|(I-K)^{\half}(Q(\varphi-h)\|^2+\|K^{\half}Qh\|^2
 < \left(( (I-K):K)Q\varphi,Q\varphi\right)+2\varepsilon.
\]
Taking the infimum over all $h\in\dom Q$ this leads to
\[
\begin{split}
& \inf
 \big\{ \|(I-K)^{\half}Q(\varphi-h)\|^2+\|K^{\half}Qh\|^2:\, h\in \dom Q \big\} \\
& \hspace{6.5cm} \leq \left(( (I-K):K)Q\varphi,Q\varphi\right).
\end{split}
\]
Therefore, equality prevails in \eqref{parsum2}
and this implies the statement by the polarization formula
for forms.
\end{proof}

Thus the representing map for the form $\sh_1:\sh_2$ is given by
$((I-K):K)^{\half}Q$, which involves the parallel sum of the operators
$I-K$  and $K$.
This result is a special case of the functional calculus developed in \cite{PW1}.
The parallel sum $\sh_1 : \sh_2$ in \eqref{ppaarr} measures the
interaction of the forms $\sh_1$ and $\sh_2$.
The forms $\sh_1$ and $\sh_2$ are called \textit{mutually singular}
if $\sh_1 : \sh_2=0$; cf. \cite[Proposition 2.10]{HSS2009}, in which case
there is no interaction
between the forms $\sh_1$ and $\sh_2$.

\begin{corollary}\label{inter0}
Under the conditions of Proposition {\rm \ref{repparsum}}
the following statements are equivalent:
\begin{enumerate} \def\labelenumi{\rm(\roman{enumi})}
\item the forms $\sh_1$ and $\sh_2$ are mutually singular;
\item $K$ is an orthogonal projection.
\end{enumerate}
\end{corollary}

\begin{proof}
It follows from  \eqref{parsum2} that $\sh_1$ and $\sh_2$ are mutually singular
if and only if $((I-K): K)^\half Q=0$.
Since the representing map $Q$ in Proposition \ref{repparsum}
is assumed to be minimal,
this is equivalent to $((I-K): K)^\half =0$. Observe that by definition
\[
(I-K):K=K-K((I-K)+K)^{-1}K=(I-K)K,
\]
cf. \cite{FW}, from which the assertion follows.
 \end{proof}

Now return to the context of Theorem \ref{Lebtype} and Theorem \ref{minim}.
The sum decompositions \eqref{lebB1} of a form $\sh$
can be classified into two different categories:
the forms $\sh_1$ and $\sh_2$ are mutually singular
(precisely when $K$ is an orthogonal projection)
or they are not mutually singular (precisely when $K$ is not an orthogonal projection);
see the remark following Theorem 2.5 in \cite{S3}.

\section{Lebesgue type decompositions of semibounded forms}\label{sec32}

The Lebesgue decomposition of a semibounded form
$\st \in \bF(\sH)$ in Section \ref{sec31}
provides a sum decomposition of $\st$ into a closable part
$\st_{\rm reg}$  and a singular part $\st_{\rm sing}$.
Recall that $\st_{\rm reg}$ is the largest of all closable
semibounded forms below $\st$. By means of the
sum decompositions in Section \ref{sec3} it is now possible
to describe the general situation.

\medskip

Based on the Lebesgue decomposition in Section \ref{sec31}
the following definition is quite natural.

\begin{definition}\label{LebTyp}
Let the form $\st \in \bF(\sH)$ be semibounded.
Then a sum decomposition of $\st$, given by
\begin{equation}\label{lebtyp}
 \st=\st_1+\st_2,  \quad \dom \st=\dom \st_1=\dom \st_2,
\end{equation}
is called a Lebesgue type sum decomposition of $\st$
if  $\st_1\in \bF(\sH)$ is semibounded and closable,
and $\st_2 \in \bF(\sH)$ is nonnegative and singular.
\end{definition}

The following characterization of all Lebesgue type decompositions
of a semibounded form is an immediate consequence of Theorem \ref{Lebtype},
Lemma~\ref{repform0}, and Lemma~\ref{repform01},
by a reduction to nonnegative forms.

\begin{theorem}\label{Lebtypep}
Let the form $\st \in \bF(\sH)$ be semibounded,
let $c \leq m(\st)$, and assume that $\st$ has the representation
\begin{equation}\label{LebB}
\st[\varphi, \psi]=c(\varphi, \psi)+(Q \varphi, Q \psi), \quad \varphi, \psi \in \dom \st=\dom Q,
\end{equation}
where $Q \in \bL(\sH, \sK)$ is a representing map for $\st-c$.

Let $K \in \bB(\sK)$ be a nonnegative contraction which satisfies
 the conditions
\begin{equation}\label{KregQ}
  \clos \{k\in\cran (I-K):\, (I-K)^\half k \in \dom Q^* \}=\cran (I-K),
\end{equation}
 \begin{equation}\label{KsingQ0}
    \ran K^{\half} \cap \dom Q^* \subset \ker Q^*,
\end{equation}
and define the forms $\st_1$ and $\st_2$ by
\begin{equation}\label{LebB2b}
\begin{array}{l}
 \st_1[\varphi, \psi]=c(\varphi, \psi)+((I-K)^{\half}Q \varphi, (I-K)^{\half}Q \psi),\\
 \st_{2}[\varphi, \psi]=(K^{\half}Q \varphi, K^{\half}Q \psi).
\end{array}
\end{equation}
 Then the sum $\st=\st_1+\st_2$ in \eqref{lebtyp}
is a Lebesgue type decomposition of $\st$ in the sense of Definition {\rm \ref{LebTyp}}.

Conversely, let the sum $\st=\st_1+\st_2$ in \eqref{lebtyp}
be a Lebesgue type decomposition of $\st$ in the sense of Definition {\rm \ref{LebTyp}}.
 Then there exists a  nonnegative  contraction $K\in \bB(\sK)$
such that \eqref{KregQ}, \eqref{KsingQ0},
and \eqref{LebB2b} are satisfied.
\end{theorem}

\begin{proof}
According to \eqref{LebB} the form $\sh=\st-c \in \bF(\sH)$ is nonnegative
with representing map $Q$. Assume $K \in \bB(\sK)$ is a nonnegative
contraction that satisfies  \eqref{KregQ} and \eqref{KsingQ0}.
Define the nonnegative forms $\sh_1, \sh_2 \in \bF(\sH)$ by
\[%begin{equation}\label{hh}
 \sh_1[\varphi, \psi]=((I-K)^{\half}Q \varphi, (I-K)^{\half}Q \psi),
 \quad
 \sh_2[\varphi, \psi]=(K^{\half}Q \varphi, K^{\half}Q \psi).
\]%end{equation}
Then by Theorem \ref{Lebtype} one has $\sh=\sh_1+\sh_2$.
The conditions \eqref{KregQ} and \eqref{KsingQ0} guarantee that
$(I-K)^\half Q$ is regular and  $K^\half Q$  is singular;
cf. \cite[Lemma 4.1, Lemma 4.3]{HS2023a}.
Hence, by Lemma~\ref{repform0} and Lemma~\ref{repform01} it follows that
$\sh_1$ and $\sh_2$ are regular and singular, respectively.
Setting $\st_1=\sh_1+c$ and $\st_2=\sh_2$, then $\st_1$ is semibounded
and $\st_2$ is nonnegative. It is now clear that
$\st=\st_1+\st_2$ is a Lebesgue type decomposition of $\st$.

For the converse, let $\st \in \bF(\sH)$ be a semibounded form
 and let it have a Lebesgue type decomposition \eqref{lebtyp}.
Since the form $\st$ is assumed to have the representation \eqref{LebB}, one notes that
\[
\st-c= \sh_1+\sh_2, \quad \sh_1=\st_1-c \geq 0, \quad \sh_2=\st_2 \geq 0,
\]
is a sum decomposition of the nonnegative form $\sh=\st-c$ into
the nonnegative forms $\sh_1$ and $\sh_2$.
By Theorem \ref{Lebtype} there exists a nonnegative contraction $K\in \bB(\sK)$
such that  \eqref{lebB2} holds. Hence, it follows that \eqref{LebB2b} is then satisfied.
Since $\st_1$ is regular and $\st_2$ is singular, it is clear that
$(I-K)^\half Q$ is regular and  $K^\half Q$  is singular.
Thus by \cite[Lemma 4.1, Lemma 4.3]{HS2023a} the conditions \eqref{KregQ}
and \eqref{KsingQ0} are satisfied.
\end{proof}

It is a consequence of Proposition \ref {repparsum} and Corollary \ref{inter0} that
the interaction between the components in a
Lebesgue type decomposition \eqref{lebtyp}  can be
specified in the following sense.

\begin{corollary}\label{kwak}
Let the conditions of Theorem {\rm \ref{Lebtypep}}  be satisfied.
Then the parallel sum $(\st_1 -c) : \st_2 \in \bF(\sH)$ has the representation
\[%begin{equation}\label{ppaarr+}
 ((\st_1-c) : \st_2)[\varphi,\psi]=\left(( (I-K):K)Q\varphi,Q\psi\right),
 \quad \varphi,\psi\in\dom \sh=\dom Q.
\]%end{equation}
In particular, the following statements are equivalent:
\begin{enumerate} \def\labelenumi{\rm(\roman{enumi})}
\item the forms $\st_1-c$ and $\st_2$ are mutually singular;
\item $K$ is an orthogonal projection.
\end{enumerate}
\end{corollary}

For the convenience of the reader, the case of orthogonal projections
in Theorem  \ref{Lebtypep} will be considered separately.
According to Corollary \ref{kwak}, in this case  there is only a trivial interaction
between the components $\st_1-c$ and $\st_2$,

\begin{corollary}\label{Lebtypepp}
Let the form $\st \in \bF(\sH)$ be semibounded,
let $c \leq m(\st)$, and let $\st$ have the representation  \eqref{LebB}
with representing map $Q \in \bL(\sH,\sK)$.
Let $P \in \bB(\sK)$ be an orthogonal projection which satisfies
the conditions
  \begin{equation}\label{comsum1ap}
  \clos ( \ker P \cap \dom Q^*) =\ker P,
\end{equation}
\begin{equation}\label{comsum1bp}
 \ran P \cap \dom Q^* \subset \ker Q^*,
\end{equation}
and define the forms $\st_1$ and $\st_2$ by
\begin{equation}\label{LebB2bp}
\begin{array}{l}
 \st_1[\varphi, \psi]=c(\varphi, \psi)+((I-P) Q \varphi, (I-P) Q \psi),\\
 \st_{2}[\varphi, \psi]=(P Q \varphi, P Q \psi).
\end{array}
\end{equation}
 Then the sum $\st=\st_1+\st_2$ in \eqref{lebtyp}
is a Lebesgue type decomposition of $\st$.
In particular. if $P=P_0$ is  the orthogonal projection onto $\mul Q^{**}$,
then \eqref{comsum1ap} and \eqref{comsum1bp} are satisfied
and \eqref{LebB2b} leads to the Lebesgue decomposition.
\end{corollary}

\begin{proof}
If $K \in \bB(\sK)$ is an orthogonal projection,
then the conditions \eqref{KregQ} and \eqref{KsingQ0}
are equivalent with the conditions
\eqref{comsum1ap} and \eqref{comsum1bp};
see \cite[Lemma 4.1, Lemma 4.2]{HS2023a}.
If $P_0$ is the orthogonal projection onto $\mul Q^{**}$, then
\[
\ran P_0=\ker (I-P_0)=\mul Q^{**}
\quad \mbox{and} \quad \ran (I-P_0)=\ker P_0=\cdom Q^*.
\]
With $P=P_0$ the conditions \eqref{comsum1ap} and \eqref{comsum1bp} are
satisfied. Comparing with \eqref{lebtyp+2} one sees that this choice corresponds
with the Lebesgue decomposition.
\end{proof}

Recall that the representing map $Q$ of a semibounded form $\st$ may always
be taken minimal. If this is the case,
then the conditions  \eqref{KsingQ0} and \eqref{comsum1bp} read as
\[%begin{equation}\label{KsingQ}
    \ran K^{\half} \cap \dom Q^* =\{0\}
    \quad \mbox{and} \quad
    \ran P \cap \dom Q^* \subset \{0\},
\]%end{equation}
respectively.
The following theorem was originally established in the context of pairs of
nonnegative forms in \cite{HSS2009}.

\begin{theorem} \label{uniq}
Let the form $\st \in \bF(\sH)$ be semibounded.
Then the following statements are equivalent:
\begin{enumerate}\def\labelenumi{\rm(\roman{enumi})}
\item  the Lebesgue decomposition of $\st$
is the only Lebesgue type decomposition of the form $\st$;
\item the form $\st_{\rm reg}$ is bounded.
\end{enumerate}
\end{theorem}

\begin{proof}
(i) $\Rightarrow$ (ii)
Assume that any Lebesgue type decomposition $\st=\st_1+\st_2$ of $\st$
is equal to the Lebesgue decomposition of $\st$.
Then, in particular, any Lebesgue type decomposition
$Q=(I-P)Q+PQ$, with $P$ an orthogonal projection
satisfying \eqref{KregQ} and \eqref{KsingQ0},
gives the  Lebesgue decomposition of $Q$.
By Theorem \cite[Theorem 6.1]{HSS2018} this implies that $Q_{\rm reg}$
is bounded, so that also the form $\st_{\rm reg}$ is bounded.

(ii) $\Rightarrow$ (i) Assume that $\st_{\rm reg}$ is bounded.
Let the form $\st$ have the representation
\[%begin{equation}\label{kuku}
\st[\varphi, \psi]=c(\varphi, \psi)+(Q \varphi, Q \psi),
\quad \varphi, \psi \in \dom \st=\dom Q,
\]%end{equation}
where $c \leq m(\st)$,
$Q$ is a representing map, and assume that $\cran Q=\sK$.
Then the regular part $Q_{\rm reg}$ is bounded or, equivalently,
$\dom Q^*$ is closed; cf. \cite[Theorem 6.1]{HSS2018} or
\cite[Theorem 5.7]{HS2023a}.
Now let $\st=\st_1+\st_2$ be a Lebesgue type decomposition of $\st$
of the form in Theorem \ref{Lebtypep} with a nonnegative contraction $K \in \bB(\sK)$.
It follows from $\mul Q^{**} \subset \ker (I-K)^\half$ or, equivalently,
$\ran (I-K)^\half \subset \dom Q^*$, that
\[
  \ran (I-K)^\half \cap \ran K^\half \subset \dom Q^* \cap \ran K^\half=\{0\},
\]
which shows that $K$ is an orthogonal projection.
This gives an orthogonal Lebesgue type decomposition $Q=(I-K)Q+KQ$ for $Q$.
Since $Q_{\rm reg}$ is bounded,  one sees that $(I-K)Q=Q_{\rm reg}$
by \cite{HSS2018}.
Hence, the Lebesgue type decomposition of $\st$ is equal to the Lebesgue decomposition.
\end{proof}

At the end of this section there is a brief discussion concerning
the connection between the Lebesgue type decompositions of a form $\st$
and the pseudo-orthogonal Lebesgue type decompositions of its representing map $Q$
(see \cite{HS2023a}). Let $K \in \bB(\sK)$ be any nonnegative contraction, then one can
write $Q$ as
\begin{equation}\label{qqqkw}
Q=Q_1+Q_2, \quad Q_1=(I-K)Q, \quad Q_2=KQ,
\end{equation}
 Recall that the decomposition \eqref{qqqkw}
is a pseudo-orthogonal Lebesgue type decomposition of $Q$
if and only if $K$ satisfies
\begin{equation}\label{condee}
 (I-K)Q \quad \mbox{is closable and} \quad KQ \mbox{ is singular}.
\end{equation}
However the nonnegative contraction $K$ generates a Lebesgue type
decomposition $\st=\st_1+\st_2$ via \eqref{LebB2bp} if and only if
\begin{equation}\label{conde}
\begin{split}
(I-K)^\half Q \mbox{ regular} \quad \mbox{and} \quad
K^\half Q  \mbox{ singular};
 \end{split}
\end{equation}
cf. \eqref{KregQ} and \eqref{KsingQ0}.
The conditions in \eqref{condee} and \eqref{conde} are
equivalent when $K$ is an orthogonal projection.
In the general case when $K$ is a nonnegative contraction
there is, for closability, only the implication
\begin{equation}\label{alf}
(I-K)Q \mbox{ is closable} \quad \Rightarrow \quad (I-K)^{\half}Q \mbox{ is closable}
\end{equation}
(use \cite[Corollary 4.2]{HS2023a}),
while both statements are equivalent with the closability of $T$ when $\|K\| <1$.
Likewise, there is, for singularity, only the implication
\begin{equation}\label{bet}
K^{\half}Q \mbox{ is singular} \quad \Rightarrow \quad KQ \mbox{ is singular}.
\end{equation}
Furthermore, it should be mentioned that, the converse statements
in \eqref{alf} and \eqref{bet} do not hold in general.
In fact, there is an  example that shows this simultaneously.

\section{Representation theorems for semibounded forms}\label{repr}

Let $\st \in \bF(\sH)$ be a semibounded form.
If $\st$ is closed then the well-known representation theorem (see \cite{BHS}, \cite{Kato})
asserts that there is a semibounded selfadjoint relation $H$ in $\sH$
with the same lower bound as $\st$, such that for all elements $\{\varphi, \varphi'\} \in H$
\[%begin{equation*}\label{repth}
\st[\varphi, \psi]=(  \varphi',  \psi)
\quad \mbox{for all} \quad  \psi \in \dom \st.
\]%end{equation*}
In this section it will be shown that for an arbitrary semibounded form $\st \in \bF(\sH)$
a similar observation can be made when the above identity is
interpreted in a limiting sense.
 Moreover, the regular part $\st_{\rm reg}$ of $\st$ is represented by the same semibounded selfadjoint
relation. This version of the representation of semibounded forms goes back
to Arendt and ter Elst (in the setting of sectorial forms); see
\cite[Theorem~3.2, Proposition~3.10]{AtE1}. The representing map for $\st \in \bF(\sH)$
allows a simple and direct exposition of the arguments.

\begin{theorem}\label{thm1}
Let the form $\st \in \bF(\sH)$ be semibounded.
Then the following statements hold:
\begin{enumerate} \def\labelenumi{\rm(\Alph{enumi})}
\item
There exists a semibounded relation $S_\st \in \bL(\sH)$
which is bounded from below by $m(\st)$, such that
$\dom S_\st \subset \dom \st$ and for each $\{\varphi, \varphi'\} \in S_\st$
 \begin{equation}\label{ppp0}
\st[\varphi, \psi]=(  \varphi',  \psi)
\quad \mbox{for all} \quad  \psi \in \dom \st.
\end{equation}
 If $H \in \bL(\sH)$ is any symmetric relation, such that
 $\dom H \subset \dom \st$ and for each $\{\varphi, \varphi'\} \in H$
 the statement \eqref{ppp0} holds,
 then $H\subset S_\st$.

\item
There exists a semibounded selfadjoint extension $\wt A_{\st} \in \bL(\sH)$ of $S_\st$
%in  $\sH$
which is bounded from below by $m(\st)$,
such that for each $\{\varphi, \varphi'\} \in \wt A_{\st}$
there exists a sequence $\varphi_n \in \dom \st$ for which
\begin{equation}\label{aa1}
\varphi_n \to_\st \varphi  \quad \mbox{and} \quad
\st[\varphi_n,\psi] \to (\varphi',\psi) \quad \mbox{for all} \quad \psi\in\dom \st.
\end{equation}
If $H \in \bL(\sH)$ is any selfadjoint relation, such that
for each $\{\varphi, \varphi'\}\in H$
there exists a sequence $\varphi_n\in\dom \st$ for which
the statement \eqref{aa1} holds, then $H=\wt A_\st$.

\item
For all $c \leq m(\st)$ the symmetric relation $S_\st$  and
the selfadjoint relation $\wt A_\st$ admit the representations
\begin{equation}\label{aq0}
 S_\st=Q_c^*Q_c+c \quad \mbox{and} \quad \wt A_\st=Q_c^*Q_c^{**}+c,
\end{equation}
where $Q_c \in \bL(\sH,\sK)$ is a representing map for $\st-c$.
In particular, the relations $S_\st$ and $\wt A_\st$
do not depend
on the choice of $c\leq m(\st)$ and the representing map $Q_c$.
 \end{enumerate}
\end{theorem}

\begin{proof}
\textrm{(A)}
Fix  $c \leq m(\st)$ and let $Q_c \in \bL(\sH, \sK)$
be some representing map for $\st-c$:
\begin{equation}\label{qqq}
 \st[\varphi, \psi]=(Q_c \varphi, Q_c \psi)+c(\varphi, \psi),
 \quad \varphi, \psi \in \dom Q_c= \dom \st.
\end{equation}
Define $S_c=Q_c^{*}Q_c+c$. Then clearly $S_c \geq c$ and $\dom S_c \subset \dom \st.$

Now let $\{\varphi, \varphi'\} \in Q_c^*Q_c$.
Then $\{Q_c\varphi, \varphi'\} \in Q_c^*$ and
if $\{\psi, Q_c \psi\} \in Q_c$, then
$(Q_c \varphi, Q_c \psi)=(\varphi', \psi)$. From \eqref{qqq} one gets
\[%begin{equation}\label{repp}
 \st[\varphi, \psi]=( \varphi',  \psi)+c(\varphi, \psi)=( \varphi'+c\varphi, \psi).
\]%end{equation}
Therefore, $S_c=Q_c^{*}Q_c+c$ satisfies \eqref{ppp0}.

Next let $H$ with $\dom H \subset \dom \st$ satisfy \eqref{ppp0}.
If $\{\varphi, \varphi'\} \in H$ and $\psi \in \dom \st$ then using \eqref{qqq} one obtains
\[
(\varphi', \psi)=(Q_c\varphi, Q_c \psi)+c(\varphi, \psi) \;\Leftrightarrow\;
(\varphi' -c \varphi, \psi)=(Q_c\varphi, Q_c \psi).
\]
Since this holds for all $\psi \in \dom Q_c$, one concludes that
\[
\{Q_c \varphi, \varphi' - c\varphi\} \in Q_c^*
\quad \mbox{or} \quad \{\varphi, \varphi' -c \varphi\} \in Q_c^*Q_c.
\]
Therefore, $H-c \subset Q_c^*Q_c$ $\Leftrightarrow$ $H\subset S_c$. 

\medskip

\textrm{(B)}
Define $A_c=Q_c^*Q_c^{**}+c$ and let $\{\varphi, \varphi'\}\in A_c$,
so that $\{\varphi, \varphi'-c \varphi\}\in Q_c^*Q_c^{**}$.
Then
$\{\varphi,g\}\in Q_c^{**}$ and $\{g, \varphi'-c \varphi \}\in Q_c^{*}$
for some $g\in \sK_c$.
Hence there exists a sequence $\varphi_n\in \dom \st$
such that $\varphi_n\to \varphi$ in $\sH$
and $Q_c \varphi_n\to g$ in $\sK_c$.
Then $(Q_c \varphi_n)$ is a Cauchy sequence and
\[
 \st[\varphi_n-\varphi_m]
 =\|Q_c (\varphi_n-\varphi_m)\|_{\sK_c}^2 +c \|\varphi_n-\varphi_m\|_{\sH}^2 \to 0,
 \quad n,m\to \infty.
\]
Thus, $\varphi_n \to_\st \varphi$.
On the other hand, for all $\psi\in \dom \st$ one has
\begin{equation}\label{eq011}
 \lim_{n\to \infty}(\st-c)[\varphi_n,\psi]
 =\lim_{n\to \infty}(Q_c \varphi_n, Q_c \psi)_{\sK_c}=(g,Q_c\psi)_{\sK_c}.
\end{equation}
It follows from $\{g,\varphi' -c \varphi\}\in Q_c^{*}$ and $\psi\in \dom Q_c$ that
\begin{equation}\label{eq0111}
 (g,Q_c\psi)_{\sK_c}=(\varphi'-c \varphi,\psi)_{\sH}.
\end{equation}
By combining \eqref{eq011} and \eqref{eq0111}, one obtains  for all $\psi\in\dom\st$,
\[
  \lim_{n\to \infty}\st[\varphi_n,\psi]
  =(\varphi'-c \varphi,\psi)_\sH+c\lim_{n\to \infty}(\varphi_n, \psi)_{\sH}=(\varphi',\psi)_{\sH}.
\]
Thus $A_c=Q_c^*Q_c^{**}+c$ satisfies \eqref{aa1}.

Next, let $H=H^*$ be such that for each $\{\varphi, \varphi'\}\in H$
there exist $\varphi_n\in\dom \st$ satisfying \eqref{aa1}.
Then $\varphi_n \to_\st \varphi$ implies that
\[
\varphi_n \to \varphi \quad \mbox{and} \quad
\|Q_c (\varphi_n-\varphi_m)\|_{\sK_c}
\to 0.
\]
Thus $(Q_c \varphi_n)$ is a Cauchy sequence in $\sK_c$ and
$Q_c \varphi_n  \to g$ for some $g \in \sK_c$.
Hence,
$\{\varphi,g\}=\lim_{n\to\infty}\{\varphi_n,Q_c \varphi_n\}\in Q_c^{**}$.
Now from the second condition in \eqref{aa1}
one gets for all $\psi\in\dom \st$
\[
 \lim_{n\to \infty}(\st-c)[\varphi_n,\psi]
 =(\varphi',\psi)_\sH - c\lim_{n\to \infty}(\varphi_n, \psi)_{\sH}
 =(\varphi'-c \varphi,\psi)_{\sH}.
\]
Thus for all $\psi\in\dom\st=\dom Q_c$,
\[
 (g,Q_c\psi)_{\sK_c}=\lim_{n\to \infty}(Q_c \varphi_n, Q_c \psi)_{\sK_c}
 =\lim_{n\to \infty}(\st-c)[\varphi_n,\psi]=(\varphi'-c \varphi,\psi)_{\sH}.
\]
Hence $\{g,\varphi'-c\varphi\}\in Q_c^*$ and, since $\{\varphi,g\}\in Q_c^{**}$,
one has
\[
\{\varphi, \varphi'-c \varphi\}\in Q_c^{*}Q_c^{**}.
\]
Thus $H-c \subset Q_c^{*}Q_c^{**}$
or, equivalently, $H \subset A_c$. Since both relations are selfadjoint,
this gives $H=A_c$. 

\medskip

\textrm{(C)}
The proofs of (A) and (B) show that $S_c=Q_c^*Q_c+c$ and $A_c=Q_c^*Q_c^{**}+c$
satisfy the desired conditions with a choice of $c\leq m(\st)$
and a representing map $Q_c$ for $\st-c$.

On the other hand, if $Q_c'=V Q_c$ with a unitary map $V$ from $\cran Q_c$ onto $\cran Q_c'$, then
\[
Q_c^*Q_c+c = (Q_c')^*Q_c'+c; \quad Q_c^*Q_c^{**}+c = (Q_c')^*(Q_c')^{**}+c.
\]
Thus, $S_c$ and $A_c$ do not depend on the choice of $Q_c$.

Finally, if $c'\leq \gamma$ and $Q_{c'}$
is some representing map for $\st-c'$, then the proof of (A) shows that
\[
 S_{c'}=Q_{c'}^*Q_{c'}+c'
\]
also satisfies the conditions in (A) and, thus, $S_{c'}\subset S_c$.
By symmetry, one also has $S_c\subset S_{c'}$. Therefore, $S_{c'}=S_c$.
This proves that $S_\st$ in \eqref{aq0} does not depend on $c \leq m(\st)$.
In particular, $S_\st=S_c \geq c$ for each $c \leq m(\st)$,
so that $S_\st \geq m(\st)$. A similar reasoning applies to $A_c$:
$\wt A_\st=Q_c^*Q_c^{**}+c$ for all $c \leq m(\st)$ and thus
$\wt A_\st \geq m(\st)$.

This completes the proof.
\end{proof}

According to Theorem \ref{thm1} any semibounded form $\st \in \bF(\sH)$
 generates a semibounded (symmetric) relation $S_\st$
and a semibounded selfadjoint relation $\wt A_\st$ in $\bL(\sH)$.
Moreover, it follows from \eqref{aq0} that they satisfy the following identities:
\begin{equation}\label{muldom}
 \mul S_\st=  (\dom \st)^\perp, \quad  \mul \wt A_\st=  (\dom \st)^\perp,
 \end{equation}
The semibounded relation $\wt A_\st$
in Theorem \ref{thm1} is selfadjoint with $\cdom \wt A_\st =\cdom \st$;
its orthogonal operator part $(\wt A_\st)_{\rm reg}$
 is given by $(\wt A_\st)_{\rm reg}=(I-P_{\rm reg}) \wt A_\st$,
where $P_{\rm reg}$ is the orthogonal projection
 from $\sH$ onto $\mul \wt A_\st$.
Therefore, in \eqref{aa1} one may replace $(\varphi', \psi)$
by  $((\wt A_\st)_{\rm reg} \varphi, \psi)$,
so that \eqref{aa1} reads
\[%begin{equation}\label{aa1+}
\varphi_n \to_\st \varphi  \quad \mbox{and}
\quad \st[ \varphi_n, \psi] \to  ((\wt A_\st)_{\rm reg} \varphi, \psi)
\quad \mbox{for all} \quad  \psi \in \dom \st,
\]%end{equation}
where $(\wt A_\st)_{\rm reg}$ is a semibounded selfadjoint operator
in $\sH \ominus \mul \wt A_\st$.
In general, the inclusion
$S_\st \subset \wt A_\st$ is not an equality:
for instance, this is the case when the form $\st$ is defined by a semibounded
relation. %, see \cite{HS2023sebst}.
In the context of closed and closable forms the main result leads to more special
results which will be discussed now.
A further discussion of the main theorem and its connections  to the case of closed
or closable forms can be found in Theorem \ref{terelst} below.

\medskip

When the form $\st \in \bF(\sH)$ in Theorem \ref{thm1} is closed
then the statements  simplify:
the relations $S_{\st}$ and $\wt A_{\st}$ coincide and will be denoted by $A_{\st}$.
Thus Theorem \ref{thm1} leads to the so-called
first representation theorem as a straightforward consequence;
cf. \cite{BHS}, \cite{Kato}.

\begin{corollary}\label{friedr}
Let the form $\st \in \bF(\sH)$ be semibounded and closed.
 Then there exists a semibounded
selfadjoint relation $A_\st \in \bL(\sH)$ with lower bound $m(\st)$,
such that $\dom A_\st \subset \dom \st$ and for each $\{\varphi, \varphi'\} \in A_\st$
 \begin{equation}\label{ppp}
\st[\varphi, \psi]=(  \varphi',  \psi)
 \quad \mbox{for all} \quad \psi \in \dom \st.
\end{equation}
If $H \in \bL(\sH)$ is a selfadjoint relation, such that
$\dom H \subset \dom \st$ and for each $\{\varphi, \varphi'\} \in H$
the statement \eqref{ppp} holds,
 then $H=A_\st$.
 Moreover, for all $c \leq m(\st)$ the selfadjoint relation $A_\st$ is given by
\begin{equation}\label{aq}
A_\st=Q_c^*Q_c+c,
\end{equation}
where the representing map $Q_c \in \bL(\sH,\sK)$ for $\st-c$ is closed.
\end{corollary}

Note that it follows from \eqref{muldom} or \eqref{aq} that the following identities
\[%begin{equation}\label{muldom1}
\mul A_\st= (\dom \st)^\perp
\quad \mbox{and} \quad \cdom A_\st= \cdom \st
\]%end{equation}
hold.
Recall that the identity \eqref{ppp} can be written
for each $\varphi \in \dom A_\st$ as
\[%begin{equation}\label{pppp1}
\st[ \varphi, \psi]= ((A_\st)_{\rm reg} \varphi, \psi)
\quad \mbox{for all} \quad  \psi \in \dom \st,
\]%end{equation}
where the orthogonal operator part $(A_\st)_{\rm reg}$
is a semibounded selfadjoint operator in $\sH \ominus \mul A_\st$.
It is now simple to see that
any semibounded selfadjoint relation in a Hilbert space $\sH$ is connected
to a closed semibounded form; cf. \cite{BHS}.

\begin{lemma}\label{converse}
Let $A \in \bL(\sH)$ be a selfadjoint relation %in a Hilbert space $\sH$
with lower bound $\gamma \in \dR$
and let $c \leq \gamma$. Then the form $\st \in \bF(\sH)$, defined by
\[%begin{equation}\label{einzz}
 \dom \st=\dom (A_{\rm reg}-c)^\half,
\]%end{equation}
and
\[%begin{equation}\label{zzwei}
 \st[\varphi, \psi] =((A_{\rm reg}-c)^\half \varphi, (A_{\rm reg}-c)^\half \psi) +c (\varphi, \psi),
 \quad
 \varphi, \psi \in \dom \st_A,
\]%end{equation}
is closed, independent of $c \leq \gamma$, and has lower bound $m(\st)=\gamma$.
Moreover, the unique semibounded selfadjoint relation in $\bL(\sH)$ associated with $\st$ is $A$.
\end{lemma}

\begin{proof}
Note that the mapping, defined by
\[
\varphi \mapsto (A_{\rm reg}-c)^\half \varphi,
\quad  \varphi \in \dom (A_{\rm reg}-c)^\half,
\]
takes $\dom (A_{\rm reg}-c)^\half$ into the Hilbert space
$\sH \ominus \mul A$ with dense range.
When $\varphi \in \dom A$, it follows that
$\st[\varphi, \psi]=(A_{\rm reg} \varphi, \psi)$.
\end{proof}

In Corollary \ref{friedr} it has been shown that for every closed
semibounded form $\st \in \bF(\sH)$
there is a unique semibounded selfadjoint relation $H=A_\st$
in $\bL(\sH)$ such that \eqref{ppp} holds.
Furthermore, according to Lemma \ref{converse},
for each semibounded selfadjoint relation $H \in \bL(\sH)$
there is a closed semibounded form $\st \in \bF(\sH)$
such that $H=A_\st$ holds. In the following, the notation
$\st(H)$ is used to denote this one-to-one correspondence.

\medskip

Let $\st_1 \in \bF(\sH)$ and $\st_2 \in \bF(\sH)$ be semibounded  forms
with representations
\[
\left\{
\begin{array}{l}
 \st_1[\varphi, \psi]=c(\varphi, \psi)+(Q_1 \varphi, Q_1 \psi), \quad \varphi, \psi \in \dom \st_1=\dom Q_1, \\
 \st_2[\varphi, \psi]=c(\varphi, \psi)+(Q_2 \varphi, Q_2 \psi), \quad \varphi, \psi \in \dom \st_2=\dom Q_2,
\end{array}
\right.
\]
with representing maps  $Q_1 \in \bL(\sH, \sK_1)$ and $Q_2 \in \bL(\sH, \sK_2)$.
Recall that $\st_1 \leq \st_2$ if and only if $Q_1 \prec_c Q_2$; see \eqref{ineq0}, \eqref{prec}, and \eqref{prec1}.
Now assume, in addition, that $\st_1$ and $\st_2$ are closed forms or, equivalently, that $Q_1$ and $Q_2$ are
closed operators, so that the corresponding semibounded selfadjoint relations $H_1$ and $H_2$ in $\bL(\sH)$
are given by
\[
H_1=Q_1^* Q_1+c, \quad H_2=Q_2^* Q_2+c.
\]
Recall from \cite[Definition 5.2.8]{BHS} that two nonnegative relations
$H_1$ and $H_2$ in $\bL(\sH)$ satisfy $H_1 \leq H_2$ when
\[%begin{equation}\label{root}
\dom H_2^\half \subset \dom H_1^\half \quad \mbox{and} \quad
\|(H_{1, {\rm reg}})^\half f\| \leq \|(H_{2, {\rm reg}})^\half\|, \,\, f \in \dom H_2^\half.
\]%end{equation}
According to \eqref{ineq0} and \cite[Theorem 2.2]{HS2023seq1} one has the equivalent statements
\[%begin{equation}\label{equequ}
\st_1 \leq \st_2 \quad \Leftrightarrow \quad Q_1 \prec_c Q_2 \quad \Leftrightarrow \quad H_1 \leq H_2.
\]%end{equation}
In particular, one has  for semibounded selfadjoint relations
$H$ and $K$ in $\bL(\sH)$ that $H \leq K$ is equivalent to
$\st(H) \leq \st(K)$ (here the above notation has been used).

\medskip

When the form $\st \in \bF(\sH)$  in Theorem \ref{thm1}
is closable the results simplify  in the sense that the description in (B)
of that theorem can be stated directly.
The result is straightforward; the main tool is Lemma \ref{closure}
in conjunction with the representations \eqref{aq0} in Theorem \ref{thm1}.

\begin{corollary}\label{closable}
Let the form $\st \in \bF(\sH)$ be semibounded and closable.
Let $\bar \st \in \bF(\sH)$ be the closure of $\st$.
Then the semibounded selfadjoint relation $A_{\bar \st} \in \bL(\sH)$
is an extension of $S_\st$ {\rm (}introduced in Theorem {\rm \ref{thm1}}{\rm )}
and it has the same lower bound $m(\st)$.
For all $c \leq m(\st)$ the semibounded relation $S_\st$ and
the selfadjoint relation $A_{\bar \st}$ admit the representations
\[%begin{equation}\label{aq000}
 S_\st=Q_c^*Q_c+c \quad \mbox{and} \quad A_{\bar \st}=Q_c^*Q_c^{**}+c,
\]%end{equation}
where $Q_c \in \bL(\sH,\sK)$ is a representing map for $\st-c$.
In particular, $A_{\bar \st}$ coincides
with the selfadjoint relation $\wt A_{\st}$ in Theorem~{\rm \ref{thm1}}.
\end{corollary}

Now return to the general context of Theorem \ref{thm1},
where $\st \in \bF(\sH)$.
Fix  $c \leq m(\st)$ and let $Q_c \in \bL(\sH,\sK)$ be some representing map for $\st-c$,
so that
\begin{equation}\label{qcqcqc}
 \st[\varphi, \psi]=c(\varphi, \psi)+(Q_c \varphi, Q_c \psi),
 \quad \varphi, \psi \in \dom Q_c= \dom \st.
\end{equation}
In the general situation of Theorem \ref{thm1}
the representing map $Q_c$ need not be closable.
However, the product $Q_c^* Q_c^{**}$ makes sense and is
a nonnegative selfadjoint relation in $\bL(\sH)$.
To proceed one turns to the Lebesgue decomposition of $Q_c$:
\[%begin{equation}\label{lebdec}
Q_c=Q_{c, \rm reg}+Q_{c, \rm sing}.
\]%end{equation}
The closable component $Q_{c, \rm reg}$
defines a closable form  $\st_{\rm reg}\in \bF(\sH)$ which in the following will be denoted by
\[
\sr=\st_{\rm reg}.
\]
It is clear that the closable form $\sr$ and its closure ${\bar \sr}$ have the representations
\begin{equation}\label{rr1}
\sr[\varphi, \psi]= c(\varphi, \psi)+(Q_{c,\rm reg} \varphi, Q_{c, \rm reg} \psi),
\quad \varphi, \psi \in \dom \sr=\dom \st,
\end{equation}
and
\begin{equation}\label{rr2}
{\bar \sr}[\varphi, \psi]= c(\varphi, \psi)+((Q_{c,\rm reg})^{**} \varphi, (Q_{c, \rm reg})^{**} \psi),
\quad \varphi, \psi \in \dom \bar \sr.
\end{equation}
Observe that one may also apply Theorem \ref{thm1} to  the closable form $\sr \in \bF(\sH)$; it turns out
that the semibounded selfadjoint relations generated by $\st$ and $\sr$ coincide.

\begin{theorem}\label{terelst}
Let the form $\st \in \bF(\sH)$ be semibounded,
let $c \leq m(\st)$, and let $\st-c$ have a representing map $Q_c \in \bL(\sH,\sK)$,
such that \eqref{qcqcqc} holds.
Let $\sr=\st_{\rm reg}$ so that \eqref{rr1} and \eqref{rr2} are satisfied.
Then the relations $S_{\sr}$ and ${\wt A}_\sr$ {\rm (}corresponding to $S_{\st}$ and ${\wt A}_{\st}$
introduced for the form $\st$ in Theorem~{\rm \ref{thm1})}
are given by
\begin{equation}\label{mmainn}
S_{\sr}=(Q_{c,\rm reg})^*Q_{c, \rm reg}+c
\quad \mbox{and} \quad
\wt A_{\sr}=(Q_{c, \rm reg})^*(Q_{c, \rm reg})^{**}+c,
\end{equation}
and they satisfy
\begin{equation}\label{mmainn+}
S_{\st} \subset S_{\sr} \quad \mbox{and} \quad
\wt A_{\st}=\wt A_{\sr}=A_{\bar \sr}.
\end{equation}
In particular, the relations $S_\sr$ and $A_{\bar \sr}$
in \eqref{mmainn} and \eqref{mmainn+} do not depend
on the choice of $c\leq m(\st)$ and the representing map $Q_c$.
 \end{theorem}

\begin{proof}
The inclusion $S_{\st} \subset S_{\sr}$ in \eqref{mmainn+} is equivalent to
$Q_c^*Q_c\subset (Q_{c,\rm reg})^*Q_{c, \rm reg}$ and the equality
$\wt A_{\st}=\wt A_{\sr}$ is equivalent to $Q_c^*Q_c^{**}=(Q_{c, \rm reg})^*(Q_{c, \rm reg})^{**}$.
This last inclusion and equality are both easily established by means of the formula
$Q_{c, \rm reg}=(I-P)Q_c$, where $P$ is the orthogonal projection onto $\mul Q_c^{**}=(\dom Q_c^{*})^\perp$;
for details, see \cite[Appendix]{HS2023seq1}.
The equality $\wt A_{\sr}=A_{\bar \sr}$ holds by Corollary \ref{closable}.
The remaining statements follow from Theorem~\ref{thm1}.
\end{proof}

Closely connected to the topics in this section are the semibounded forms
induced by semibounded operators or relations.
Their representing maps can be used to define the extremal extensions, including
the Friedrichs extension and the Kre\u{\i}n type extension; cf. \cite{AN}, \cite{A88}, \cite{AHSS}, \cite{BHS}, \cite{HMS};
this provides a connection with the work of Sebesty\'en and Stochel and their coworkers \cite{CS}, \cite{PS}, \cite{SS}, \cite{ST},
which will be discussed elsewhere.

\section{Monotone sequences of semibounded forms}\label{mono}

Nondecreasing sequences of semibounded forms
have a limit and, likewise, nonincreasing sequences
of semibounded forms with a common lower bound have a limit;
see for instance \cite{Kato}, \cite{RS1}, \cite{S3}.
In this section the convergence of monotone sequences
of semibounded forms will be considered in connection with
the convergence of the corresponding representing maps
as in Section \ref{repr}.

\medskip

First some general facts are established.
Let $\st_n \in \bF(\sH)$  be a sequence of semibounded forms
whose lower bounds are uniformly bounded:
\[%begin{equation}
 \gamma \leq m(t_n) \quad \mbox{for some} \quad \gamma \in \dR,
\]%end{equation}
and let $c \leq \gamma$. Then
there exists a sequence of representing maps
$Q_n \in \bL(\sH, \sK_n)$, where $\sK_n$ are Hilbert spaces,
such that
\begin{equation}\label{seq}
\st_n[\varphi, \psi]=c+(Q_n \varphi, Q_n \psi),
\quad \varphi, \psi \in \dom \st_n=\dom Q_n.
\end{equation}
Conversely, each sequence of linear operators
$Q_n \in \bL(\sH, \sK_n)$
defines via \eqref{seq} a sequence
of semibounded forms $\st_n \in \bF(\sH)$
such that $m(\st_n) \geq c$.
Now assume that the sequence of
semibounded forms $\st_n \in \bF(\sH)$ satisfies
\begin{equation}\label{grijpq0}
\st_m \leq \st_n, \quad m \leq n.
\end{equation}
In this case one can take $\gamma=m(\st_1)$ and the
representing maps $Q_n \in \bL(\sH, \sK_n)$ in \eqref{seq}
satisfy
\begin{equation}\label{ggrijpa}
Q_m \prec_c Q_n, \quad m \leq n.
\end{equation}
Conversely, if $Q_n \in \bL(\sH, \sK_n)$ satisfies \eqref{ggrijpa} then the
semibounded forms $\st_n \in \bF(\sH)$ in \eqref{seq} satisfy \eqref{grijpq0}.
It is clear from \eqref{seq} that
\[
 \bigcap_{n=1}^\infty \dom \st_n=  \bigcap_{n=1}^\infty \dom Q_n,
\]
and, moreover, for an element $\varphi$ in this set one has
\[
 \sup_{n \in \dN} \st_n[\varphi] < \infty
 \quad \Leftrightarrow \quad  \sup_{n \in \dN} \|Q_n \varphi\| < \infty.
\]
Recall that if \eqref{ggrijpa} is satisfied, then there exists a linear operator
$Q \in \bL(\sH,\sK)$, where $\sK$ is a Hilbert space, which  satisfies
  \begin{equation}\label{Grijpq}
 Q_n \prec_c Q \quad \mbox{and} \quad
 \|Q_n \varphi\| \nearrow \|Q \varphi \|, \quad \varphi \in \dom Q,
 \end{equation}
 where $\dom Q$ is given by
\[%begin{equation}\label{grijpbq}
  \dom Q=\left\{\varphi \in
  \bigcap_{n \in \dN} \dom Q_n :\, \sup_{n \in \dN} \|Q_n \varphi]\|< \infty \right\},
\]%end{equation}
see \cite[Theorem 5.1]{HS2023seq1}.
The linear operator $Q \in \bL(\sH, \sK)$
serves as a representing map for the semibounded form $\st \in \bF(\sH)$ defined by
\begin{equation}\label{seqq}
\st[\varphi, \psi]=c+(Q \varphi, Q \psi), \quad \varphi, \psi \in \dom \st=\dom Q.
\end{equation}
Hence the following lemma, going back to Simon (see \cite{S3} and \cite{RS1})
is now straightforward.

\begin{lemma}\label{garijp1}
Let $\st_n \in \bF(\sH)$  be a sequence of semibounded forms,
represented  as in \eqref{seq}. Assume that the sequence
satisfies
\begin{equation}\label{grijpq}
\st_m \leq \st_n, \quad m \leq n.
\end{equation}
Then there exists a unique semibounded form $\st \in \bF(\sH)$, represented in \eqref{seqq},
such that
\[%begin{equation}\label{grijpb}
  \dom \st=\left\{\varphi \in
  \bigcap_{n \in \dN} \dom \st_n :\, \sup_{n \in \dN} \st_n[\varphi] < \infty \right\}
\]%end{equation}
and which satisfies
\[%begin{equation}\label{grijpc}
 \st_n \leq \st  \quad \mbox{and} \quad
 \st_n[\varphi] \nearrow \st[\varphi],
 \quad  \varphi \in \dom \st.
\]%end{equation}
Furthermore, let $\,\su \in \bF(\sH)$ be a semibounded form. Then there is the implication
\[%begin{equation}\label{grijpd}
 \st_n \leq \su, \quad n \in \dN \quad \Rightarrow \quad \st \leq \su.
\]%end{equation}
The following statements hold:
\begin{enumerate} \def\labelenumi{\rm(\alph{enumi})}
\item  if \,$\st_n$ is closable for all
$n \in \dN$, then $\st$ is closable;
\item if \,$\st_n$ is closed for all  $n \in \dN$,
then $\st$ is closed.
\end{enumerate}
If  $\st_n$  is a bounded everywhere defined form
for all $n \in \dN$, then $\st$
is a bounded everywhere defined form.
\end{lemma}

In the limit procedure of Lemma \ref{garijp1} the properties of
being closable and closed are preserved, respectively. This observation
will be used in the following discussion of the regular parts  of $\st_n$
and  $ \st$ in $\bF(\sH)$.
It follows from \eqref{seq} and \eqref{seqq} that $\sr_n=\st_{n, {\rm reg}}$
and $\sr=\st_{{\rm reg}}$ have the representations
 \begin{equation}\label{seqr}
 \sr_n[\varphi, \psi]=c+(Q_{n, \rm reg} \varphi, Q_{n, \rm reg} \psi),
\quad \varphi, \psi \in \dom \sr_n=\dom \st_{n}=\dom Q_n,
\end{equation}
and
\begin{equation}\label{seqqr}
 \sr[\varphi, \psi]=c+(Q_{\rm reg} \varphi, Q_{\rm reg} \psi),
\quad \varphi, \psi \in \dom \sr=\dom \st =\dom Q.
\end{equation}
The assumption  $Q_m \prec_c Q_n$ in \eqref{ggrijpa} implies that
\begin{equation}\label{aap0}
Q_{m, {\rm reg}} \prec_c Q_{n, {\rm reg}}, \quad m \leq n,
\end{equation}
and, likewise, the inequality $Q_n \prec_c Q$ in \eqref{Grijpq} gives
\begin{equation}\label{aapp0}
Q_{n, {\rm reg}} \prec_c Q_{\rm reg};
\end{equation}
cf. \eqref{presreg}.
In particular, $Q_{\rm reg}$ is an upper bound for $Q_{n, {\rm reg}}$.
By \eqref{aap0} and Lemma \ref{garijp1}
it follows from the closability of the operators $Q_{n, {\rm reg}}$ that there exists
a closable operator $Q_{\rm r} \in \bL(\sH, \sK')$ such that its domain
is given by
  \begin{equation}\label{grijpbs0}
  \dom Q_{\rm r}=\left\{\varphi \in
  \bigcap_{n \in \dN} \dom Q_{n}  :\,
  \sup_{n \in \dN} \|Q_{n, {\rm reg}} \varphi]\|< \infty \right\}
\end{equation}
and which satisfies
\begin{equation}\label{Grijpqs}
\left\{
\begin{array}{l} Q_{n, {\rm reg}} \prec_c Q_{\rm r} \prec Q_{\rm reg}, \\
 \|Q_{n, {\rm reg}} \varphi\| \nearrow \|Q_{\rm r} \varphi \|,
 \quad \varphi \in \dom Q_{\rm r}.
\end{array}
\right.
\end{equation}
The closable operator $Q_{\rm r}$ serves as a representing map
for the closable semibounded form $\st_{\rm r} \in \bF(\sH)$ defined by
\begin{equation}\label{aap}
\st_{\rm r} [\varphi, \psi]=c(\varphi, \psi)+(Q_{\rm r} \varphi, Q_{\rm r} \psi),
\quad \varphi, \psi \in \dom Q_{\rm r}.
\end{equation}

\medskip

Since the regular parts $Q_{n, {\rm reg}}$ and $Q_{\rm r}$ are closable,
the semibounded forms $\sr_{n}$ and $\st_{\rm r}$ are closable, and one
obtains from \eqref{seqr} that
\begin{equation}\label{seqa}
%\begin{split} &
 \bar \sr_n[\varphi, \psi]
  =c+((Q_{n, \rm reg})^{**} \varphi, (Q_{n, \rm reg})^{**} \psi)
%  \\  &\hspace{5cm}  \varphi, \psi \in \dom \bar \sr_n   =\dom (Q_{n, \rm reg})^{**},
%\end{split}
\end{equation}
for all $\varphi, \psi \in \dom \bar \sr_n =\dom (Q_{n, \rm reg})^{**}$
and from \eqref{aap} that
\begin{equation}\label{seqqreg}
 (\clos \st_{{\rm r}}) \,[\varphi, \psi]
 =c+((Q_{\rm r})^{**} \varphi, (Q_{\rm r})^{**} \psi)
%\quad \varphi, \psi \in \dom (\clos \st_{\rm r})=\dom (Q_{\rm r})^{**}.
\end{equation}
for all $\varphi, \psi \in \dom (\clos \st_{\rm r})=\dom (Q_{\rm r})^{**}$.
It follows from the inequalities \eqref{aap0} that
\begin{equation}\label{aap1}
(Q_{m, {\rm reg}})^{**} \prec_c (Q_{n, {\rm reg}})^{**}, \quad m \leq n,
\end{equation}
and from \eqref{aapp0} that
\begin{equation}\label{aapp1}
(Q_{n, {\rm reg}})^{**} \prec_c (Q_{{\rm r}})^{**}.
\end{equation}
By \eqref{aap1} and Lemma \ref{garijp1}
it follows  that there exists a closed operator $S_{\rm r}\in \bL(\sH, \sK'')$
such that its domain is given by
  \begin{equation}\label{grijpbs}
  \dom S_{\rm r}=\left\{\varphi \in
  \bigcap_{n \in \dN} \dom (Q_{n, {\rm reg}})^{**} :\,
  \sup_{n \in \dN} \|(Q_{n, {\rm reg}})^{**} \varphi]\|< \infty \right\},
\end{equation}
and which satisfies
\begin{equation}\label{Grijpqss}
\left\{
\begin{array}{l}
 (Q_{n, {\rm reg}})^{**} \prec_c S_{\rm r} \prec_c (Q_{\rm reg})^{**}, \\
\|(Q_{n, {\rm reg}})^{**} \varphi\| \nearrow  \|S_{\rm r} \varphi \|,
\quad \varphi \in \dom S_{\rm r}.
\end{array}
\right.
\end{equation}
The closed operator  $S_{\rm r}$ serves
as a representing map for the closed semibounded form
$\ss_{\rm r} \in \bF(\sH)$ defined by
\begin{equation}\label{seqclor}
\ss_{\rm r}[\varphi, \psi]=c+(S_{\rm r} \varphi, S_{\rm r} \psi),
\quad \varphi, \psi \in \dom \ss_{\rm r}=\dom S_{\rm r}.
\end{equation}
It is clear from the above that  $S_{\rm r}  \prec (Q_{\rm r})^{**}$.

\medskip

The above facts together with Lemma \ref{garijp1} now lead to the following theorem,
that ultimately describes the limit behaviour of the semibounded selfadjoint
relations corresponding to $\st_n$ as described in Theorems~\ref{thm1},~\ref{terelst}.

\begin{theorem}%\label{mani}
Let  $\st_{n} \in \bF(\sH)$  be a sequence of semibounded forms,
represented in \eqref{seq}, which satisfies   \eqref{grijpq}.
Let the semibounded form  $\st \in \bF(\sH)$,
 represented in \eqref{seqq}, be the limit of $\st_n$.
Then the following statements hold:
\begin{enumerate} \def\labelenumi{\rm(\Alph{enumi})}
\item The regular parts $\sr_n  \in \bF(\sH)$, represented in \eqref{seqr}, satisfy
\begin{equation}\label{grijpa1}
 \sr_m \leq  \sr_n, \quad m \leq n,
 \quad \mbox{and} \quad \sr_n \leq   \st_{{\rm reg}},
 \end{equation}
where $\st_{\rm reg}$ is represented in \eqref{seqqr}.
Moreover, there is a closable semibounded form $\st_{\rm r} \in \bF(\sH)$,
represented in \eqref{aap}, such that
\begin{equation}\label{grijpb1}
 \dom \st_{\rm r}
 =\left\{\varphi \in \bigcap_{n \in \dN} \dom \st_n :\,
   \sup_{n \in \dN}  \sr_n [\varphi] < \infty \right\},
\end{equation}
 and which satisfies
\begin{equation}\label{gggrijpc1}
  \sr_n \leq \st_{\rm r}  \leq \st_{\rm reg}
  \quad \mbox{and} \quad
  \sr_n [\varphi] \nearrow \st_{\rm r}[\varphi], \quad
 \varphi \in \dom \st_{\rm r}.
 \end{equation}

\item
The closures $ \bar \sr_n \in \bF(\sH)$ of the regular parts $\sr_n$,
represented in \eqref{seqa}, satisfy
\begin{equation}\label{ggrijpa1}
 \bar \sr_m \leq  \bar \sr_n   \quad m \leq n,
  \quad \mbox{and} \quad  \bar \sr_n \leq \clos \,\st_{\rm r},
 \end{equation}
where  $\clos\, \st_{\rm r}$ is represented in \eqref{seqqreg}.
Moreover, there exists a closed semibounded form $\ss_{\rm r} \in \bF(\sH)$,
represented in \eqref{seqclor},
such that
\begin{equation}\label{ggrijpb1}
  \dom \ss_{\rm r}=\left\{\varphi \in \bigcap_{n \in \dN} \dom \bar \sr_n
  :\,
  \sup_{n \in \dN} \,   \bar \sr_n[\varphi] < \infty \right\},
\end{equation}
and which satisfies
\begin{equation}\label{ggrijpcc1}
 \bar \sr_n \leq \ss_{\rm r} \leq \clos \,\st_{\rm r} \quad \mbox{and} \quad
  \bar \sr_n [\varphi] \nearrow \ss_{\rm r}[\varphi], \quad  \varphi \in \dom \ss_{\rm r}.
  \end{equation}
In fact, $ \dom (\clos \st_{\rm r}) \subset \dom \ss_{\rm r} $ and
 \begin{equation}\label{grijpdd2}
  (\clos \st_{\rm r} )\,[\varphi, \psi]=\ss_{\rm r}[\varphi, \psi],
 \quad \varphi, \psi \in \dom (\clos \,\st_{\rm r}).
\end{equation}

\item
The semibounded selfadjoint relations $\wt A_{\st_n} \in \bL(\sH)$
corresponding to the semibounded forms $\st_n \in \bF(\sH)$
converge to the semibounded selfadjoint relation $A_\ss \in \bL(\sH)$
corresponding to the closed semibounded form $\ss \in \bF(\sH)$:
\begin{equation}\label{grijp111}
 \wt A_{\st_n} \to A_\ss
\end{equation}
in the strong resolvent sense or, equivalently,  the strong graph sense in $\sH$.
\end{enumerate}
\end{theorem}

\begin{proof}
Due to the assumption \eqref{grijpq} it follows for the corresponding
representing maps that $Q_m \prec_c Q_n$, $m \leq n$, and that $Q_n \prec_c Q$.

\medskip

(A) The inequalities in \eqref{grijpa1} follow from \eqref{aap0} and \eqref{aapp0}.
The statements about $\st_{\rm r}$ in \eqref{grijpb1} and \eqref{gggrijpc1}
follow from \eqref{grijpbs0} and \eqref{Grijpqs}.

\medskip

(B) The inequalities in \eqref{ggrijpa1} follow from \eqref{aap1} and \eqref{aapp1}.
The statements about $\ss_{\rm r}$ in \eqref{ggrijpb1} and \eqref{ggrijpcc1}
follow from \eqref{grijpbs} and \eqref{Grijpqss}.
The equality \eqref{grijpdd2} holds by polarization, after observing that $\st_{\rm r}[\varphi]$
is the limit of $\sr_n [\varphi]$ for $\varphi \in \dom \st_{\rm r}$ while $\ss_{\rm r}[\varphi]$
is the limit of $\bar \sr_n [\varphi]$ for $\varphi \in \dom \ss_{\rm r}$,
see \eqref{gggrijpc1} and \eqref{ggrijpcc1}. This equality is then preserves also for
the closure $\clos \st_{\rm r}$, since $\st_{\rm r}\subset \ss_{\rm r}$ and $\ss_{\rm r}$ is closed.

\medskip

(C) The semibounded selfadjoint relation $\wt A_n \in \bL(\sH)$
corresponding to $\st_n \in \bF(\sH)$ via Theorem~\ref{thm1} is given by
 \[
 \wt A_{n}=c+(Q_n)^* (Q_n)^{**}
 =c+(Q_{n, \rm reg})^*(Q_{n, \rm reg})^{**},
\]
where the second equality follows again from 
$(Q_n)^* (Q_n)^{**}=(Q_{n, \rm reg})^*(Q_{n, \rm reg})^{**}$, cf. \cite[Appendix]{HS2023seq1}.
Hence, by the classical representation theorem for closed forms, $\wt A_{n}$   
is the unique semibounded selfadjoint relation corresponding to the closed form $\bar \sr_n$. 
The sequence $(Q_{n, \rm reg})^{**}$ satisfies \eqref{Grijpqss} while \eqref{ggrijpcc1} 
shows that the sequence $\bar \sr_n$ has the closed limit form $\ss_{\rm r}$ with 
the closed representing operator $S_{\rm r}$ in \eqref{seqclor}. 
One concludes that $A_\ss=(S_{\rm r})^*S_{\rm r}$ and now the strong resolvent or, equivalently, 
the strong graph convergence in \eqref{grijp111} follows from the standard monotonicity principle
for closed forms; cf. e.g. \cite[Theorem~5.2.15]{BHS}. 
\end{proof}

Finally let $\st_n \in \bF(\sH)$, $n \in \dN$, be a sequence
of semibounded forms which is  nonincreasing:
\begin{equation}\label{grijpaa}
 c < \st_n \leq \st_m, \quad m \leq n,
\end{equation}
in the sense of \eqref{ineq0}. The assumption of the common
lower bound $c \in \dR$ guarantees the existence of a limit.
Due to \eqref{grijpaa} the lower bounds satisfy
\[
c \leq m(\st_n) \leq m(\st_m),  \quad m \leq n.
\]
The following result is straightforward,
see \cite[Theorem 10.3]{HS2023seq1}, \cite{S3}, \cite{RS1}.

\begin{lemma}%\label{garijp11}
Let $\st_n \in \bF(\sH)$  be a sequence of semibounded forms
which satisfies \eqref{grijpaa}.
Then  there exists a semibounded form $\st \in \bF(\sH)$
such that
\begin{equation}\label{grijpcc}
  \dom \st=\bigcup_{n \in \dN} \dom \st_n,
\end{equation}
 and which satisfies
 \begin{equation}\label{grijpcq}
  c \leq
 \st \leq \st_n  \quad \mbox{and} \quad
 \st_n[\varphi] \searrow \st[\varphi],
 \quad  \varphi \in \dom \st.
 \end{equation}
\end{lemma}

In the case of nonincreasing sequences the notions of closability
or closedness are in general not preserved; see \cite[Example 10.5]{HS2023seq1}, \cite{RS1}.
There is a useful  result for nonincreasing sequences of closed forms
which goes back to \cite{S3}; see also \cite{RS1}. The following
result is included for completeness: it is the analog of 
\cite[Theorem 10.4]{HS2023seq1},  
when adapted to the setting of nonincreasing sequences of forms.

\begin{theorem}%\label{grijp4}
Let $\st_n \in \bF(\sH)$, represented as in \eqref{seq},
 be a sequence of closed semibounded forms
such that \eqref{grijpaa} holds.
Let the semibounded form $\st \in \bF(\sH)$,
represented as in \eqref{seqq}, be the limit of $\st_n$ as
in \eqref{grijpcc} and \eqref{grijpcq}.
Then the semibounded selfadjoint relations $A_{\st_n} \in \bL(\sH)$
corresponding to $\st_n \in \bF(\sH)$
converge to a semibounded selfadjoint relation $A_\infty \in \bL(\sH)$:
\[
 A_{\st_n} \to A_\infty,
\]
in the strong resolvent sense or, equivalently,
in the strong graph sense in $\sH$.
Let $\st_\infty \in \bF(\sH)$ be the closed semibounded
form corresponding to $A_\infty \in \bL(\sH)$.
Then the semibounded forms $\st$ and $\st_\infty$ are connected by
\[
 \clos \st_{\rm reg}=\st_\infty.
\]
Moreover, for the semibounded form $\st \in \bF(\sH)$ one has
\begin{enumerate} \def\labelenumi{\rm(\alph{enumi})}
\item $\st$ is closable if and only if $\st \subset \st_\infty$;
\item $\st$ is closed if and only if $\st=\st_\infty$;
\item $\st$ is singular if and only if $A_\infty-c$ is singular.
\end{enumerate}
\end{theorem}

%\begin{proof}
%It follows from \eqref{???} that $\st_\infty  \leq \st_n$.
%However, this implies with \eqref{grijpcc} and \eqref{grijpcq}
%that $\st_\infty \leq \st \leq \st_n$, where $\st_\infty$ and $\st_n$ are closed.
%Therefore one obtains also \marginpar{need reference?}
%\[
% \st_\infty \leq \st_{\rm reg} \leq \st_n,  \quad \mbox{so that also } \quad
%  \st_\infty \leq \clos \st_{\rm reg} \leq \st_n,
%\]
%which gives for $\gamma < c$ that
%\[
% (A_n-\gamma)^{-1} \leq (A-\gamma)^{-1} \leq (A_\infty-\gamma)^{-1}, \quad n \in \dN.
%\]
%where $A$ is the semibounded selfadjoint relation corresponding to $\clos \st_{\rm reg}$.
%Thus the desired result follows by taking limits.
%
%If there is the inclusion $\st \subset \st_\infty$,
%then clearly $\st$ is closable; conversely, if $\st$ is closable,
%then $\st=\st_{\rm reg} \subset \clos \st_{\rm reg}=\st_\infty$.
%Likewise, if $\st=\st_\infty$, then $\st$ is closed; conversely,
%if $\st$ is closed then $\st=\clos \st_{\rm reg}=\st_\infty$.
%\end{proof}


\begin{thebibliography}{33}

\bibitem{An}
T.~Ando,
``Lebesgue-type decomposition of positive operators'',
Acta Sci. Math. (Szeged), 38 (1976), 253--260.

\bibitem{AN}
T.~Ando and K.~Nishio,
''Positive selfadjoint extensions of positive symmetric operators'',
Toh\'oku Math. J., 22 (1970), 65--75.

\bibitem{AtE1}
W. Arendt and T. ter Elst,
``Sectorial forms and degenerate differential operators'',
J. Operator Theory, 67 (2012), 33--72.

\bibitem{A88}
Yu.M. Arlinski\u{\i},
``Positive spaces of of boundary values and sectorial extensions of nonnegative symmetric operators”,
Ukrainian Math. J., 40 (1988), 8–15.

\bibitem{AHSS}
Yu.M.~Arlinski\u{\i}, S.~Hassi, H.S.V.~de~Snoo, and Z.~Sebesty\'en,
``On the class of extremal extensions of a nonnegative operator'',
Oper. Theory: Adv. Appl. (B.~Sz.-Nagy memorial volume),
127 (2001), 41--81.

\bibitem{BHS}
J.~Behrndt, S.~Hassi, and H.S.V.~de~Snoo,
 \textit{Boundary value problems, Weyl functions, and differential operators},
 Monographs in Mathematics, Vol. 108, Birkh\"auser, 2020.

%\bibitem{C}
%E.A.~Coddington,
%Extension theory of formally normal and symmetric subspaces,
%Mem. Amer. Math. Soc., 134, 1973.

\bibitem{CS}
E.A.~Coddington and H.S.V.~de~Snoo,
``Positive selfadjoint extensions of positive symmetric subspaces'',
Math. Z., {159} (1978), 203--214.

%\bibitem{Doug}
%R.G. Douglas,
%``On majorization, factorization and range inclusion of operators in Hilbert space'',
%Proc. Amer. Math. Soc. 17 (1966), 413-416.

\bibitem{FW}
P.A. Fillmore and J.P. Williams,
``On operator ranges",
Adv. Math. 7 (1971), 254--281.

%\bibitem{H}
%S.~Hassi,
%``On the Friedrichs and the Kre\u{\i}n-von Neumann extension
%of nonnegative relations'',
%Acta Wasaensia 122 (2004), 37--54.

\bibitem{HMS}
S.~Hassi, M.M.~Malamud, and H.S.V. de~Snoo,
``On Kre\u{\i}n's extension theory of nonnegative operators'',
Math. Nachr., 274/275 (2004), 40--73.

\bibitem{HSS2009}
S.~Hassi, Z.~Sebesty\'en, and H.S.V. de~Snoo,
``Lebesgue type decompositions for nonnegative forms'',
J. Functional Analysis, 257 (2009), 3858--3894.

\bibitem{HSS2018}
S.~Hassi, Z.~Sebesty\'en, and H.S.V. de~Snoo,
``Lebesgue type decompositions for linear relations and Ando's uniqueness criterion",
Acta Sci. Math. (Szeged), 84 (2018), 465--507.

%\bibitem{HS2015}
%S. Hassi and H. S.V. de Snoo,
%``Factorization, majorization, and domination for linear relations'',
%Annales Univ. Sci. Budapest, 58 (2015), 53--70.

\bibitem{HS2023a}
S.~Hassi  and H.S.V. de~Snoo,
``Complementation and Lebesgue type sum decompositions of linear operators and relations",
submitted for publication.

\bibitem{HS2023seq1}
S.~Hassi  and H.S.V. de~Snoo,
``Sequences of operators, monotone in the sense of contractive domination",
submitted for publication.

%\bibitem{HS2023sebst}
%S.~Hassi  and H.S.V. de~Snoo,
%``Friedrichs and Kre\u{\i}n type extensions in terms of representing maps'',
%%``Semibounded relations?? Sebesty\'en Stochel'',


 \bibitem{Kato}
T.~Kato, \textit{Perturbation theory for linear operators},
Springer-Verlag, Berlin, 1980.

\bibitem{Kosh}
V.D.~Koshmanenko,
\textit{Singular quadratic forms in perturbation theory},
Kluwer Academic Publishers, Dordrecht/Boston/London,
Mathematics and its applications, Vol. 474, 1999.

\bibitem{PS}
V. Prokaj and Z. Sebesty\'en,
``On Friedrichs extensions of operators”,
Acta Sci. Math. (Szeged), 62 (1996), 243–246.

\bibitem{PW1}
W.~Pusz and S.L.~Woronowicz,
"Functional calculus for sesquilinear forms and the purification map",
Rep. Math. Phys. 5 (1975), 159--170.

\bibitem{RS1}
M. Reed and B. Simon.
\textit{Methods of modern physics. I.}
Academic Press, New York, 1980.


\bibitem{SS}
Z.~Sebesty\'en and J.~Stochel,
``Restrictions of positive self-adjoint operators'',
Acta Sci. Math. (Szeged), 55 (1991), 149--154.

\bibitem{ST}
Z.~Sebesty\'en and Z.~Tarcsay,
``Extensions of positive symmetric operators and Krein's uniqueness criteria'',
arXiv, 2022.

\bibitem{S3}
B.~Simon,
``A canonical decomposition for quadratic forms with applications
to monotone convergence theorems'',
J. Functional Analysis, 28 (1978), 377--385.

\bibitem{Szy87}
W. Szyma\'nski,
``Positive forms and dilations",
Trans. Amer. Math. Soc., 301 (1987),  761--780.

\end{thebibliography}
\end{document}